\documentclass[english,10pt]{article}

\usepackage[english]{babel}
\usepackage[utf8x]{inputenc}
\usepackage[T1]{fontenc}
\usepackage[formats]{listings}
\usepackage{geometry}
\usepackage{enumitem}

\usepackage{graphicx}
\usepackage[colorinlistoftodos]{todonotes}
\usepackage{setspace}
\usepackage{placeins}
\usepackage{enumitem}
\usepackage{titlesec}
\usepackage{cite}
\usepackage{comment}
\usepackage{caption}
\usepackage{subcaption}
\usepackage{enumitem}
\usepackage{courier}

\usepackage[colorlinks=true, allcolors=blue]{hyperref}

\usepackage{amsmath}
\usepackage{amsthm}
\usepackage{amssymb}
\usepackage{amsfonts}
\usepackage{bbm}
\usepackage{bm}
\usepackage{nicefrac}
\usepackage{mathtools}

\usepackage{tikz}

\usepackage{graphicx}
\usepackage{fullpage}
\usepackage{float}
\usepackage{wrapfig}
\usepackage{caption}

\usepackage{algpseudocode}
\usepackage{algorithm}
\usepackage{listings}

\usepackage{amsthm}
\usepackage{dsfont}
\usepackage{array}
\usepackage{mathrsfs}
\usepackage{cite}
\usepackage{comment}
\usepackage{mathrsfs}

\makeatother

\usepackage{babel}
\usepackage{color}
\usepackage{centernot}

\usepackage{babel}
\usepackage{sansmath}

\newcommand{\R}{\mathbb{R}}

\newcommand{\E}{{\mathbb E}}

\def\argmin{\mathop{\rm argmin}}

\def\Leb{\mathrm{Leb}}
\def\Unc{\mathrm{Unc}}
\def\Cert{\mathrm{Cert}} 
\def\Unif{\mathrm{Unif}}

\title{Regret Minimization in Isotonic, Heavy-Tailed Contextual Bandits via Adaptive Confidence Bands}

\author{Sabyasachi Chatterjee\thanks{Department of Statistics, UIUC} \and 
  Subhabrata Sen\thanks{Department of Statistics, Harvard
    University, Cambridge, MA 02138, U.S.A.} }

\theoremstyle{plain}\newtheorem{lemma}{\textbf{Lemma}}\newtheorem{theorem}{\textbf{Theorem}}\newtheorem{proposition}{\textbf{Proposition}}

\theoremstyle{definition}

\theoremstyle{definition}\newtheorem{remark}{\textbf{Remark}}

\makeatletter
\renewenvironment{proof}[1][\proofname] {
	\par\pushQED{\qed}\normalfont
	\topsep6\p@\@plus6\p@\relax
	\trivlist\item[\hskip\labelsep\bfseries#1\@addpunct{:}]
 	\ignorespaces
} {
	\popQED\endtrivlist\@endpefalse
}
\makeatother

\begin{document}

\maketitle

\begin{abstract}
    In this paper we initiate a study of non parametric contextual bandits under shape constraints on the mean reward function. Specifically, we study a setting where the context is one dimensional, and the mean reward function is isotonic with respect to this context. We propose a policy for this problem and show that it attains minimax rate optimal regret. Moreover, we show that the same policy enjoys automatic adaptation; that is, for subclasses of the parameter space where the true mean reward functions are also piecewise constant with $k$ pieces, this policy remains minimax rate optimal simultaneously for all $k \geq 1.$ Automatic adaptation phenomena are well-known for shape constrained problems in the offline setting;
    we show that such phenomena carry over to the online setting. 
    The main technical ingredient underlying our policy is a procedure to derive confidence bands for an underlying isotonic function using the isotonic quantile estimator. The confidence band we propose is valid under heavy tailed noise, and its average width goes to $0$ at an adaptively optimal rate. We consider this to be an independent contribution to the isotonic regression literature. 
\end{abstract}

\section{Introduction}
\label{sec:introduction}


The Multi-Armed Bandit (MAB) problem \cite{thompson1933likelihood,robbins1952some}  is a widely studied model for sequential decision making. Motivated by diverse applications such as clinical trials and recommendation systems, this problem has been extensively explored across statistics, machine learning, operations research, etc. The MAB problem exhibits the well-known exploration/exploitation tradeoff---the decision maker must balance the desire to maximize reward based on current information with the need to explore alternatives, with the promise of higher future rewards. We refer the interested reader to \cite{slivkins2019introduction,bubeck2012regret} for a textbook introduction. 

Contextual bandits present a generalization of the MAB problem, where the decision-maker observes additional information about the usefulness of each action.  The contextual bandit problem with stochastic contexts is particularly relevant in clinical trial applications, where the context can model the demographic features of each individual. It is natural to believe that upon utilizing these contexts, one could design policies with improved regret guarantees.

From a technical standpoint, it is clear that the intrinsic hardness of the problem is governed by the complexity of the mean reward function  for each arm. The simplest assumption in this regard is to posit a parametric (usually linear) model on these mean reward-functions---we point the interested reader to \cite{slivkins2019introduction} for some seminal results in this setting. In diverse applications, the parametric assumption on the mean-reward might be too restrictive, and a non-parametric approach might be desirable. Nonparametric contextual bandit problems have been studied recently, mostly under smoothness assumptions on the mean-reward functions (we survey these results in detail in Section~\ref{subsec:current art}). Further, the traditional literature on the contextual bandit problem assumes bounded/subgaussian reward distributions for each arm. However, modern applications in finance, e-commerce etc. routinely encounter heavy-tailed data, and this motivates a thorough study of contextual bandits with heavy-tailed reward distributions.

In this paper, we initiate a study of a non-parametric contextual bandits problem with shape-constraints on the mean reward functions and heavy-tailed reward distributions. Specifically, we assume that the mean-reward functions are monotonic. Such an assumption is very common in the clinical trials setting. Suppose the arms represent two competing drugs, and the context represents the age of the patient. In this case, it is natural to assume that the efficacy of the drug would deteriorate with increasing age.

First, we delineate our formal setup. 
Consider a contextual bandit problem with two arms, referred to as $0$ and $1$. At the $i^{th}$ step, we see the context $X_i$, taking values in the  context space $\mathcal{X} = [0,1]$. Depending on the history and the current context, we draw one of the arms $A_i \in \{0,1\}$, and see the reward $Y_i(A_i)$ corresponding to the arm drawn. Throughout, we assume that $\{(X_i, Y_i(0), Y_i(1)): i \geq 1\}$ are iid, and we denote the arm drawn at the $i^{th}$ step as $A_i$. 


For simplicity of exposition, we assume that 
$X_i \sim \Unif(0,1)$ iid random variables.  Our results generalize easily to the setting of non-uniform distributions, as long as the density of contexts is uniformly bounded away from zero. 
%
Set 
\begin{align}
\mathcal{F}_{[0,1]\uparrow} = \{ F:[0,1] \to \mathbb{R}: F \,\textrm{monotone non-decreasing}, 0 \leq F(x) \leq 1\}. \label{def:parameter-space} 
\end{align} 
Let us denote by $\mathcal{D}(\tilde{C},L)$ the set of all possible joint distributions on $(X,Y(0),Y(1))$ which satisfy
\begin{itemize}
    \item[(i)] $X \sim \Unif(0,1)$. 
    \item[(ii)] Fix $F_0, F_1 \in \mathcal{F}_{[0,1]\uparrow}$. We assume 
    \begin{align}
        Y(0) = F_0(X) + \varepsilon, \,\,\,\, Y(1) = F_1(X) + \varepsilon, \nonumber 
    \end{align}
    where $\varepsilon$ satisfies $\mathbb{E}[\varepsilon]=0$, and it's distribution is symmetric around zero. Further, we assume that $\varepsilon$ and $X$ are independent. 
    Thus given $X$, $Y(0)$ and $Y(1)$ have isotonic conditional means $F_0$, $F_1$ respectively, and their distribution is symmetric around the conditional mean functions. Due to symmetry, given $X$, the conditional median functions of $Y(0)$ and $Y(1)$ are also $F_0$, $F_1$ respectively. We emphasize that $\varepsilon$ can be heavy-tailed, and not have any moment beyond it's expectation.

    \item[(iii)] The error distribution $\varepsilon$  satisfies Assumption A \eqref{eq:assum} with parameters $\tilde{C}$ and $L$. 
\end{itemize}


\noindent 
Assumption A pertains to the local growth of $F_0, F_1$ around zero---we will assume throughout that $\tilde{C}$ and $L$ are known to the statistician. Such local growth conditions are standard in the quantile regression literature and are quite mild (see Remark~\ref{rem:growth} for further discussion of this point).  

The set of distributions $\mathcal{D}(\tilde{C},L)$ form the data generating model or the parameter space. In other words, we assume that there is one distribution in $\mathcal{D}(\tilde{C},L)$ which generates the i.i.d sequence $\{ (X_i , Y_i(0), Y_i(1)): i \geq 1\}.$ 
Given any two functions $F_0, F_1$, denote 
\begin{align}
F_*(x) = \max\{ F_0(x), F_1(x) \}. \nonumber 
\end{align}

We wish to design a policy $\pi$ which performs ``optimally". To make sense of this, define $\mathscr{F}_{t} = \sigma( \{X_i, Y_i(A_i) : i \leq t\})$. A policy $\pi$ will refer to a scheme where $A_{t+1}$ is a $\{0,1\}$-valued function measurable with respect to $\sigma( \mathscr{F}_t, X_{t+1})$. We compare any policy to the oracle-optimum policy, which knows the optimal arm for each context. The oracle optimum policy has expected payoff $\sum_{i=1}^{T} \mathbb{E}[F_*(X_i)]$, and thus we define 
\begin{align}
\mathrm{Regret}_T(\pi) = \sum_{i=1}^{T} \mathbb{E} \Big[ F_*(X_i) - Y_i(A_i) \Big]. \nonumber 
\end{align}
\noindent
We design a policy $\pi^*$ which attains the optimal rate for the worst case regret (see Theorem~\ref{thm:upperinform} below for an informal statement). This policy can be thought of as a functional version of the Successive Elimination algorithm which is a standard Multi Armed Bandit algorithm (see Chapter $2$ of Slivkins). We provide a a very brief high level description of the optimal policy below. 
The full details of our policy are given in Section~\ref{sec:policy}.
\begin{enumerate}
	\item[(i)] The policy proceeds in epochs, and at each epoch, it forms confidence bands for the two isotonic conditional mean/median functions $F_0,F_1$.
	
	\item[(ii)] After each epoch, based on the constructed bands, we isolate a part of the context space $\mathcal{X}$ where the bands are non-overlapping. In subsequent epochs, this part becomes a part of the \textit{exploitation} region, and we draw the arm corresponding to the dominant estimate. The complement of this set remains in the \textit{exploration} region. 
\end{enumerate}


\noindent
The two main results about our policy $\pi^*$ are presented below informally. Our first theorem establishes that the policy $\pi^*$ is minimax rate optimal. 
\begin{theorem}[Informal]
\label{thm:upperinform}
We have, 
\begin{align*}
\max_{D \in \mathcal{D}(\tilde{C},L)} \mathrm{Regret}_T(\pi^*) \leq \tilde{O}(T^{2/3}), \nonumber 
\end{align*}
where $\tilde{O}(\cdot)$ hides log-terms in $T$. Moreover, we have 
\begin{equation*}
\min_{\pi} \max_{D \in \mathcal{D}(\tilde{C},L)} \mathrm{Regret}_T(\pi) \geq \Theta(T^{2/3}).
\end{equation*}
\end{theorem}

\noindent 
In practice, the minimax criteria might be too conservative, and it might be possible to achieve substantially smaller worst case regret if the parameters $F_0, F_1$ are restricted to a subset of the parameter space. This naturally motivates the question of \textit{adaptation}---is the policy $\pi^*$ optimal over certain simpler sub-classes? 

Our next result establishes that the policy $\pi^*$ is adaptively minimax optimal over the sub-class of piecewise constant monotone functions. To this end, let us now define for any positive integer $k \geq 1$,
\begin{align}
\mathcal{F}^{(k)}_{[0,1]\uparrow} = \{ F:[0,1] \to \mathbb{R}: F \,\textrm{monotone non-decreasing with k constant pieces}, 0 \leq F(x) \leq 1\}. \label{def:parameter-space} 
\end{align} 
\noindent 
For each positive integer $k \geq 1$, the above defines a sub parameter space $\mathcal{D}^{(k)}(\tilde{C},L) \subset \mathcal{D}(\tilde{C},L)$ where both the conditional mean functions $F_0,F_1 \in \mathcal{F}^{(k)}_{[0,1]\uparrow}.$ 

\begin{theorem}[Informal]
	\label{thm:upperinform}
	We have, 
	\begin{align*}
	\max_{D \in \mathcal{D}^{(k)}(\tilde{C},L)} \mathrm{Regret}_T(\pi^*) \leq \tilde{O}( \sqrt{kT}), \nonumber 
	\end{align*}
	where $\tilde{O}(\cdot)$ hides log-terms in $T$. Moreover, we have 
	\begin{equation*}
	\min_{\pi} \max_{D \in \mathcal{D}^{(k)}(\tilde{C},L)} \mathrm{Regret}_T(\pi) \geq O(\sqrt{kT}).
	\end{equation*}
\end{theorem}

\subsection{Confidence Bands for Isotonic Quantile Functions}
We construct adaptive confidence bands for isotonic median functions at each epoch of our policy. Construction of such confidence bands is a statistical question of natural interest, and has been relatively unexplored in the existing literature.  
We develop novel, finite sample valid, optimally rate adaptive confidence bands for isotonic quantile (for a general quantile level $0 < \tau < 1$) regression with heavy tailed errors. This is one of our main technical contributions in this paper. 

We first provide an informal glimpse of our confidence bands to whet the appetite of the reader. 
Suppose $f:[0,1] \rightarrow [0,1]$ is a nondecreasing/isotonic function. Consider the setting where $x_1,\dots,x_n$ i.i.d $\Unif(0,1)$ and we observe $y_i = f(x_i) + \epsilon_i$. For a given $0 < \tau < 1$, $\epsilon_i$  are i.i.d with CDF $F$ with $F(0) = \tau.$ This implies that the $\tau$ quantile of $y_i$ is $f(x_i)$. We sometimes refer to the $\epsilon$ variables as \textit{errors}. We do not assume any tail decay/moment assumptions on the errors. Under this setting we consider the problem of deriving upper and lower confidence bands for $f$; i.e two data dependent functions $\hat{U},\hat{L} : [0,1] \rightarrow \mathbb{R}$ such that for a given coverage probability $\alpha$ we have 
\begin{equation*}
\mathbb{P}(\hat{L}(x) \leq f(x) \leq \hat{U}(x) \:\:\forall x \in [0,1]) \geq 1 - \alpha.
\end{equation*}

\noindent 
Our confidence band method is based on the isotonic $\tau$ quantile estimator function $\hat{f}.$ This estimator is defined as follows:
\begin{equation*}
\hat{f} = \argmin_{f \in \mathcal{F}_{[0,1]\uparrow}} \sum_{i = 1}^{n} \rho_{\tau}(y_i - f(x_i))
\end{equation*} 
where $\rho_{\tau}: \R \rightarrow \R$ is the piecewise linear convex function $x \rightarrow x(\tau - \mathbb{1}(x < 0)).$  Note that the description above specifies the optimizer only on the design points. In fact, even on the design points, the definition is non-unique---our results are valid for any minimizer of this objective. Subsequently, we extrapolate the function to the interval $[0,1]$ using a natural piecewise constant extrapolation scheme. This ensures that the final estimate $\hat{f}$ is piecewise constant. 
%
%
%
%
%
%
Informally, the confidence band we propose in this paper takes the following form:

\begin{equation}\label{eq:bandus}
[\hat{L}(x), \hat{U}(x)] = \Big[\hat{f}(x) -  \frac{q_{\alpha} \sqrt{\log n}}{\sqrt{x - \hat{l}(x)}}, \hat{f}(x) +  \frac{q_{\alpha} \sqrt{\log n}}{\sqrt{\hat{u}(x) - x}} \Big]
\end{equation}
where $[\hat{l}(x),\hat{u}(x)]$ is the constant piece of $\hat{f}$ containing $x.$ In the above display, $q_{\alpha}$ is a specific number only depending on $\tilde{C},L$ and thus is pivotal for a given parameter space $\mathcal{D}(\tilde{C},L).$

The confidence band above is meant to provide an initial glimpse into our band construction. It is morally correct but is not entirely accurate as our construction involves additional details. In the actual construction, we first fit the isotonic quantile estimator $\hat{f}$ on the design points. This fit is non-decreasing, and thus piecewise constant. We use the recipe described above to construct the confidence bands on the design points. However, we have to modify the construction at design points which lie near the edge of the constant blocks in the estimated $\hat{f}$. This modification is needed to ensure the finite sample validity of the confidence bands in the presence of heavy-tailed errors. Finally, the band is extended to the interval [0,1] using a natural piecewise-constant, monotonic extension. The full details are given in Section~\ref{subsec:band_construction}.
 


\noindent 
The following result collects our main theoretical guarantees regarding the confidence band. 
\begin{theorem}[Informal]\label{thm:widthinform}
	The confidence band functions $L,U$ satisfy the following two properties:
	\begin{itemize}
		\item \textbf{Finite Sample Coverage:} We have $P(L(x) \leq f(x) \leq U(x) \:\:\forall x \in [0,1]) \geq 1 - \alpha.$
		
		\item \textbf{Adaptive Rate Optimal Width:} If $Z \sim \Unif(0,1)$ is a new test point independently drawn from the training data $\{x_i,y_i\}_{i = 1}^{n}$; then we have $$\E \:\left(U(Z) - L(Z)\right) \lesssim 
		\min\Big\{\frac{1}{n^{1/3}},\sqrt{\frac{k}{n}} \Big\}$$
		where $k$ is the number of constant pieces of the underlying true quantile function $f.$
	\end{itemize}
\end{theorem}




\subsection{Background and Comparisons with Existing Literature}\label{subsec:current art}

\begin{itemize}
    \item[(i)] Contextual bandits --- The contextual bandit problem is known to be a useful midway between the classical multi-armed bandit problem with iid rewards, and the adversarial bandit setting. Contexts are common in applications arising in clinical trials, e-commerce, finance etc., and this has motivated an extensive investigation into variants of the contextual bandit problem. We refer the interested reader to Chapter $8$ of \cite{slivkins2019introduction} for an introduction to this model and some classical results. The challenge of this problem is governed by the complexity of the conditional mean functions. Traditionally, one assumes a parametric form on the conditional means.  Currently, there is an emerging line of work which studies bandit problems with ``high-dimensional" contexts (see e.g. \cite{abbasi2013online,bastani2020online,ren2020dynamic,wang2018minimax,kim2019doubly,carpentier2012bandit}).  These developments exploit recent advances in high-dimensional statistics to design appropriate policies which are applicable in this setting.

     Recently, there has been substantial interest in non-parametric contextual MAB problems, which impose weaker structural assumptions on the mean-reward functions.~\cite{rigollet2010nonparametric} formulated a non-parametric model for the contextual bandit problem with two arms, and assuming smoothness  of the conditional response functions, characterized the minimax Regret in this setting. These results were generalized in \cite{perchet2013multi} to the setting of $k$ arms. These results assumed that the conditional response functions were Holder smooth with index $\beta<1$. 
     Recently, \cite{hu2020smooth} extended this result, and introduced a family of regret-optimal policies for any $\beta \in [1, \infty)$. Unfortunately, the optimal policies derived in these papers assume access to the true smoothness $\beta$ of the underlying conditional response functions. This is rarely true in practice; this prompts the question of adaptation. \textit{Is is possible to design policies that achieve the optimal regret, without explicit knowledge of the underlying smoothness?}.  This question was explored in \cite{gur2019smoothness}, who discovered that this is \emph{not possible} in general. They established that the absence of  adaptive policies arises from the non-existence of adaptive confidence sets in non-parametric regression under smoothness classes. Subsequently, they leveraged recent breakthroughs in non-parametric regression to establish that adaptive policies exist under additional \emph{self-similarity} assumptions on the mean reward functions.

    These existing works on nonparametric contextual bandits assume Holder smoothness of the mean reward functions. The corresponding problem assuming shape constraints on the mean reward functions such as monotonicity, convexity etc. has not been studied at all in the literature. Such shape constraints are arguably more natural than Holder smoothness constraints in several applications. This paper initiates a study of this problem.  
    We establish that when the conditional mean reward functions are isotonic, there exist globally minimax rate optimal policies that adapt automatically to the sub-parameter space of piecewise constant functions. Such auto-adaptation phenomena are well-understood in the offline setting of nonparametric regression under isotonic constraints; see~\cite{chatterjee2015risk},~\cite{chatterjee2018matrix},~\cite{han2019isotonic}. We exhibit that similar phenomena carry over to the online setting.  
    \item[(ii)] Bandits with heavy tails --- In the multi-armed bandit literature, one typically assumes that the reward distributions are bounded and sub-gaussian. This assumption is restrictive in applications with heavy-tailed data, arising e.g. in financial applications. Another concern is with data corruptions. \cite{bubeck2013bandits,liu2011multi} are early attempts at mitgating this problem. More recently, rapid progress has been achieved in combining recent advances in algorithmic high-dimensional robust statistics and online learning (see e.g.  \cite{lee2020optimal,lu2019optimal, liu2011multi, yu2018pure, shao2018almost, medina2016no, wei2020minimax,dubey2020cooperative,agrawal2021regret, tao2021optimal,altschuler2019best,chen2020online}).  This prompts us to consider heavy tailed noise in contrast to the usual  subgaussian assumption. 
    
    
    \item[(iii)] Thresholding Bandits under shape constraints---Recently, \cite{cheshire2020influence} has initiated a study of the thresholding bandit problem under shape constraints on the mean-rewards of the arms. In this problem, one wishes to identify the arms with mean-rewards above a certain level. This problem is more closely related to the Multi-armed Bandit problem, and the corresponding insights do not seem directly related to the problem studied in this paper.

    \item[(iv)] Isotonic Quantile regression --- 
    The Isotonic Least Squares Estimator (LSE) has a long history in Statistics and has been thoroughly studied from several aspects, see Section $2.1$ of~\cite{groeneboom2014nonparametric} for a textbook reference. The Isotonic Quantile Regression (IQR) estimator studied in this paper is the quantile version of the Isotonic LSE. The IQR estimator has been studied far less compared to its LSE counterpart. The IQR estimator appears to have been first proposed by~\cite{cryer1972monotone}. Pointwise limiting behaviour of this estimator is known, with the cube root $O(n^{-1/3})$ asymptotic rate of convergence; 
    see~\cite{abrevaya2005isotonic},~\cite{van1990estimating}. Apart from this, nothing much else appears to have been established about the IQR estimator. 

    Coming to the question of constructing confidence bands in isotonic regression, there appears to be only a handful of papers in the literature giving rigorously valid confidence bands. A multi scale testing based confidence band method was pioneered by~\cite{dumbgen2003optimal} for shape constrained mean regression. This multi scale testing approach to build confidence bands was then generalized to shape constrained quantile regression; see~\cite{dumbgen2004confidence} and~\cite{dumbgen2007confidence}. More recently,~\cite{yang2019contraction} gives a confidence band method based on Isotonic LSE for mean regression.

    Our confidence band method for an underlying isotonic quantile function is different from the multiscale testing based method of~\cite{dumbgen2003optimal}. 
    The main difference is that our band is based on the IQR estimator whereas the traditional bands are based on kernel estimators. Second, the confidence band of \cite{dumbgen2003optimal} is only valid asymptotically whereas our band is valid for finite samples. The bands in~\cite{dumbgen2003optimal} also rely on Monte Carlo simulations which make it computationally heavier than our band which can be computed in $O(n)$ time; see Remark~\ref{rem:comp}.

    The flavor of our band is more similar to the bands for isotonic regression given in~\cite{yang2019contraction}. Their bands, which are based on the Isotonic LSE, are also finite sample valid and use concentration inequalities (as do we) to construct the band; see Theorem $3$ in~\cite{yang2019contraction}. However, since our band is based on the IQR estimator, our band is more robust to heavy tailed errors. In fact, the band given in~\cite{yang2019contraction} is valid only under sub gaussian errors. In comparison, our band remains valid even when the errors have no finite first moment like the Cauchy distribution.

A major point of difference in our analysis here with~\cite{dumbgen2003optimal},~\cite{yang2019contraction} is that both do not give the rigorous adaptive width guarantee as in Theorem $3$. To the best of our knowledge, such an adaptive width guarantee result for a isotonic confidence band method is new. Theorem $3$ establishes that adaptive inference is possible for IQR and consequently adaptive estimation is also possible for IQR. To place Theorem $3$ in context, it is necessary to mention that several recent papers have established an automatic adaptation property of Shape Constrained Least Squares Estimators; see the survey paper~\cite{guntuboyina2017adaptive} and references therein. The main theme of this line of research is that the shape constrained  LSE often automatically adapts (exhibiting faster rates of convergence) to certain kinds of \textit{sparsity} in the parameter space without explicit regularization. For instance, asymptotic risk bounds are known for the Isotonic LSE; see~\cite{Zhang02},~\cite{chatterjee2015risk}. These bounds not only give the cube root rate of convergence of the Isotonic LSE for a global loss function like the mean squared error (MSE) but also establish a near parametric rate $\tilde{O}(k/n)$ rate of convergence of the same Isotonic LSE estimator when the true underlying function is piecewise constant with $k$ pieces in addition to being monotone. This adaptivity of the shape constrained LSE estimator needs no external regularization which is why the adaptation is said to be automatic.

The automatic adaptation property of shape constrained LSE for estimation raises the following two questions. 
\begin{itemize}
    \item Do Shape Constrained Quantile Regression Estimators also enjoy adaptive estimation guarantees?
    \item Is Adaptive Inference also possible under Shape Constraints? For example, does the width of the confidence intervals/bands based on shape constrained estimators adapt?
\end{itemize}

The two questions above appear to be largely open research directions for general shape constraints. Theorem $3$ makes a new contribution to the isotonic regression literature by answering the above two questions in the affirmative in the setting of univariate isotonic regression.

\item[(v)] Using Local Information for Inference---The form of our confidence bands bears some similarities to the recent work of~\cite{deng2020confidence}. They consider the problem of pointwise inference for an isotonic mean function at a given point $x \in [0,1]$ say. They construct an asymptotically valid confidence interval for $f(x)$ which is of the form 

\begin{equation}\label{eq:isobandcunhui}
I = \Big[\hat{f}(x) - q_{\alpha} \frac{\sigma}{\sqrt{\hat{w}(x)}},\hat{f}(x) + q_{\alpha} \frac{\sigma}{\sqrt{\hat{w}(x)}} \Big]
\end{equation}
where $\hat{f}$ is the isotonic LSE function (which is piecewise constant by nature), $\hat{w}(x)$ is the length of the constant piece of $\hat{f}$ containing $x_0$ and $q_{\alpha}$ is the $1 - \alpha$ quantile of a pivotal distribution which is related to the Chernoff distribution~\cite{Groeneboom89Airy},~\cite{groeneboom2001computing}. The novelty of this approach, compared to earlier approaches, is in using the local information $\hat{w}(x)$ in the width of the confidence interval. The asymptotic distribution of $n^{1/3} \big(\hat{f}(x) - f(x)\big)$ is known from classical results in shape constrained regression literature (see Section $4$ in~\cite{guntuboyina2018nonparametric} and references therein) and these results can obviously be used to derive asymptotic confidence intervals for $f(x).$ However, the asymptotic confidence interval constructed following this classical approach would involve several nuisance parameters. In contrast, the new approach in~\cite{deng2020confidence} does not suffer from this issue.

In our setting, we require confidence bands and not just confidence intervals for a single point $x_0 \in [0,1].$ We also require finite sample coverage since we are interested in applications to Contextual Bandits. Moreover, we base our confidence band on the IQR estimator and not the Isotonic LSE. Despite these differences, the similarity here is in the fact that both their confidence interval and our confidence bands use information about the random interval which is the constant piece containing $x$ in the piecewise constant estimators Isotonic LSE and IQR. Note that compared to~\eqref{eq:isobandcunhui}, our band uses the additional information about the position of $x$ within $[\hat{l}(x_0),\hat{u}(x)]$ and not just its length $\hat{w}(x) = \hat{u}(x) - \hat{l}(x).$ The fact that by using this slightly more detailed information one may obtain confidence bands has not been explicitly realized in the shape constrained regression literature so far.

\item[(vi)] Computation--- It is well known that the Isotonic LSE estimator can be computed by the Pooled Adjacent Violators Algorithm (PAVA) in $O(n)$ time; see~\cite{mair2009isotone} and references therein.
Versions of the PAVA algorithm exist for the Isotonic Median estimator which can be computed in $O(n \log n)$ time; see~\cite{stout2008isotonic}. It has recently been established that the Isotonic Quantile Regression estimator (for a general quantile $\tau$) can be computed in $O(n \log n)$ time by dynamic programming; see~\cite{hochbaum2017faster}. This implies that our entire confidence band can actually be computed in $O(n \log n)$ time; see Remark~\ref{rem:comp}.

\end{itemize}

\noindent 
\textbf{Notation:} Throughout this paper, we use the standard Landau notation $O(\cdot)$, $o(\cdot)$, and $\Theta(\cdot)$ for asymptotics. We say that $a_n \lesssim b_n$ if there exists some constant $C>0$ such that $\limsup a_n/b_n \leq C$. Sometimes, we abuse notation, and use $\lesssim$ to also drop polylog factors. We use $C,C'$ as universal constants throughout the paper---these will be positive constants independent of the problem parameters. The precise constants $C,C'$ will change from line to line.  

\noindent
\textbf{Outline:} The rest of the paper is structured as follows. We describe our policy and collect our main result on its worst case regret in Section~\ref{sec:results}. We describe our confidence band construction and the related guarantees in Section~\ref{sec:confidence_sets}. Section~\ref{sec:discussions} discusses some complements to our main results, and collects some directions for follow up research. In Section \ref{sec:sims}, we collect some numerical simulations exploring the finite sample performance of our confidence bands.  Finally, Section~\ref{sec:proofs} contains the proofs of the main results.

\section{Policy Description and Regret Bounds}
\label{sec:results}

\subsection{Policy}
\label{sec:policy}
We introduce a policy $\pi^*$ for our problem for a horizon of $T$ rounds. The policy proceeds in epochs, and in epoch $i$, we observe $N_i$ new datapoints. We set $N_1 = \lceil \sqrt{T} \rceil$ and $N_{i + 1} = 2 N_i$ for $i > 1$. We leave it implicit that the last epoch is potentially of a smaller size. This means that the total number of epochs $K$ is at most $\log_{2}(1 + \frac{T}{N_1}) - 1.$ Also, set $\alpha_{T} = T^{-2}.$  We describe the epochs of this policy concretely below. 

\textbf{Epoch 1:} We observe $N_1$ datapoints. Denote the corresponding contexts as $X_1, \cdots, X_{N_1} \sim \Unif(0,1)$ iid random variables. At this point, we have no information, so we sample the arms $\{A_l : 1\leq l \leq N_1\}$ iid uniform from $\{0,1\}$, and we observe the corresponding outcomes $Y_l(A_l)$. We define 
\begin{align}
\mathcal{S}^{(1)}_j = \{ l \in [N_1]: A_l =j \} , \,\,\, j=0,1\}. \nonumber 
\end{align}
We fit a separate isotonic regression on each of the datasets $\{X_l,Y_l(A_l)\}_{l \in \mathcal{S}_j^{(1)}}$ for $j \in \{0,1\}$, and construct $1 - \alpha_{T}$ confidence bands $U_j^{(1)}$ and $L_j^{(1)}$ for $ j \in \{0,1\}.$
At this point, we will be able to determine the optimal arm in certain parts of the space. To this end, we define the sets
\begin{align}
\Cert_0^{(1)} = \{x \in [0,1] : L^{(1)}_0(x) > U^{(1)}_1(x)\}. \nonumber \\
\Cert_1^{(1)} = \{x \in [0,1] : L^{(1)}_1(x) > U^{(1)}_0(x)\}. \nonumber 
\end{align} 
Finally, we set $\Unc^{(1)} = [0,1] \cap [\Cert_0^{(1)} \cup \Cert_1^{(1)}]^c$. This completes the operations in Epoch 1. 

\textbf{Epoch 2:} We observe $N_2$ new observations---denote the corresponding contexts as $X_{N_1 +1}, \cdots, X_{N_1+N_2}$. We proceed in the following steps: 
\begin{itemize}
	\item If $X_{N_1 + l} \in \Cert_j^{(1)}$ for $j \in \{0,1\}$, we pull arm $j$ and observe $Y_{N_1 + l}(j)$. If $X_{N_1 + l} \in \Unc^{(1)}$, pull the arm $A_{N_1 + l}$ iid uniform from $\{0,1\}.$
	
	 \item We define for $j \in \{0,1\}$,
	\begin{align}
	\mathcal{S}^{(2)}_j = \{l \in [N_1 + 1,N_2]: X_{l} \in \Unc^{(1)},A_l =j \}. \nonumber 
	\end{align}
	
	\item We fit a separate isotonic regression on each of the datasets $\{X_l,Y_l(A_l)\}_{l \in \mathcal{S}_j^{(2)}}$ for $j \in \{0,1\}$, and construct $1 -  \alpha_{T}$ confidence bands $U_j^{(2)}$ and $L_j^{(2)}$ for $j \in \{0,1\}.$
	
	\item At this point, we have identified an additional part of the context space where one of the arms dominates. We define the sets 
	\begin{align}
	\Cert^{(2)}_0 = \{ x \in \Unc^{(1)}: L^{(2)}_0(x) > U^{(2)}_1(x) \} , \nonumber \\
	\Cert^{(2)}_1 = \{ x \in \Unc^{(1)}: L^{(2)}_1(x) > U^{(2)}_0(x) \}. \nonumber 
	\end{align}
\end{itemize}
Finally, we set $\Unc^{(2)} = \Unc^{(1)} \cap [\Cert^{(2)}_0 \cup \Cert^{(2)}_1]^{c}$. 
This completes the operations in Epoch 2. 

\textbf{Epoch $i$:} For notational convenience, for $m \geq 1$, we set $\bar{N}_m = \sum_{k=1}^{m} N_k$. At this step, we observe $N_i$ new observations---we denote the corresponding contexts as $X_{\bar{N}_{i-1} + 1} , \cdots, X_{\bar{N}_i}$. We proceed as follows:
\begin{itemize}
	\item If $X_{N_{i - 1} + l} \in \cup_{m = 1}^{i - 1} \Cert_j^{(m)}$ for $j \in \{0,1\}$, we pull arm $j$ and observe $Y_{N_{i - 1} + l}(j)$. If $X_{N_{i - 1} + l} \in \Unc^{(i - 1)}$, pull the arm $A_{N_{i - 1} + l}$ iid uniform from $\{0,1\}.$
	
	\item  We define for $j \in \{0,1\}$,
	\begin{align}
	\mathcal{S}^{(i)}_j = \{l \in [N_{i - 1} + 1,N_i]: X_{l} \in \Unc^{(i - 1)},A_l =j \}. \nonumber 
	\end{align}

	\item Fit isotonic regressions separately to the datasets $\{X_l,Y_l(A_l)\}_{l \in \mathcal{S}_j^{(i)}}$ for $j \in \{0,1\}$, and construct $1 - 2 \alpha_{T}$ confidence bands $U_j^{(i)}$ and $L_j^{(i)}$ for $ j \in \{0,1\}.$

	\item At this point, we have identified an additional part of the context space where one of the arms dominates. We define the sets
	\begin{align}
	\Cert^{(i)}_0 = \{ x \in \Unc^{(i - 1)}: L^{(i)}_0(x) > U^{(i)}_1(x) \} , \nonumber \\
	\Cert^{(i)}_1 = \{ x \in \Unc^{(i - 1)}: L^{(i)}_1(x) > U^{(i)}_0(x) \}. \nonumber 
	\end{align}
\end{itemize}
Finally, we set $\Unc^{(i)} = \Unc^{(i - 1)} \cap [\Cert^{(i)}_0 \cup \Cert^{(i)}_1]^{c}$. This completes the operations in Epoch $i$.

\begin{remark}
Note that for any epoch $i \geq 1$, we have $\Unc^{(i )} \subseteq \Unc^{(i - 1)}$. Also, note that the two sets $\cup_{j = 1}^{i} \big[\Cert_{0}^{(j)} \cup \Cert_{1}^{(j)}\big]$ and $\Unc^{(i)}$ are disjoint and their union is the whole of the context space $[0,1].$ 
\end{remark}

\begin{remark}
At each step, the policy moves certain sub-intervals between consecutive datapoints from $\Unc^{(i-1)}$ to $\cup_{j=1}^{i} \Cert_0^{(j)}\cup \Cert_1^{(j)}$. As a result, note that the set $\Unc^{(i)} \subseteq [0,1]$ is a union of finite disjoint intervals for all $ i\geq 1$. Another way to explain this is our confidence band functions are piecewise constant with finite number of pieces and hence at every epoch $i$, the set $\Cert^{(i)}$ remains a finite union of intervals.  
\end{remark}

\subsection{Regret Analysis}
The following result is our main regret bound on the policy $\pi^*$ described in Section~\ref{sec:policy}.
\begin{theorem}\label{thm:regret}
There exists an absolute constant $C>0$ such that 
\begin{align}
    \max_{D \in \mathcal{D}(\tilde{C},L)} \mathrm{Regret}_T(\pi^*) \leq C T^{2/3} (\log T)^{3/2}, \,\,\, \max_{D \in \mathcal{D}^{(k)}(\tilde{C},L)} \mathrm{Regret}_T(\pi^*) \leq C (\sqrt{k} + 1) \sqrt{T} (\log T)^2. \nonumber   
\end{align}

\end{theorem}

\noindent 
The following lemma shows that our policy $\pi^*$ attains minimax rate optimal regret. 
 \begin{lemma}
 \label{lem:lower}
 There exists a universal constant $c>0$ such that for all integer $k \geq 1$ 
 \begin{align}
 \min_{\pi} \max_{D \in \mathcal{D}(\tilde{C},L)} \mathrm{Regret}_T(\pi) \geq c T^{2/3}, \,\,\,\,\,  \min_{\pi} \max_{D \in \mathcal{D}^{(k)}(\tilde{C},L)} \mathrm{Regret}_T(\pi) \geq c \sqrt{k T} .\nonumber 
 \end{align}
 \end{lemma} 
 
 \subsubsection{Overview of regret upper bound} We now give a sketch of the proof of Theorem~\ref{thm:regret} which bounds the regret of our policy $\pi^*.$ This sketch is presented for the convenience of the reader and is meant to convey the essence of our proof strategy for bounding the regret of our policy. The complete proof is deferred to Section \ref{sec:upper_bound_proof}. 

First, we decompose our total regret into regrets $\{R_i:1 \leq i \leq K\}$ incurred during each of the $K$ epochs as below:

\begin{equation*}
\mathrm{Regret}_{T}(\pi^*) = \sum_{i \in [K]} R_i = \sum_{i \in [K]} \:\sum_{l = \bar{N}_{i-1}+1 }^{\bar{N}_i}\mathbb{E}[ F_*(X_{l}) - F_{A_l}(X_l)],
\end{equation*}
where the last equality follows by taking the expectation of $Y_l(A_l)$ given the past and $A_l$. 

We now fix an epoch $i \in [K]$ and $l \in [\bar{N}_{i-1}+1,\bar{N}_i]$ and bound $\mathbb{E}[ F_*(X_{l}) - F_{A_l}(X_l)].$ 
\begin{itemize}
	\item Step 1: Define the event $$G_{i - 1} =  \cap_{j = 0}^{1}\{U_{j}^{(i - 1)}(x) \geq F_{j}(x) \geq L_{j}^{(i - 1)}(x)\:\:\forall x \in [0,1]\}.$$

In words, $G_{i - 1}$ is the event that the confidence bands formed for $F_0,F_1$ after epoch $i - 1$ actually cover $F_0,F_1$. Since $\alpha_{T} = T^{-2}$ the event $G_{i - 1}^{c}$ has small enough probability so that we can ignore it and hence it suffices to bound 
\begin{equation*}
\mathbb{E}[ F_*(X_{l}) - F_{A_l}(X_l)] \mathbb{1}(G_{i - 1}) = \mathbb{E}[ F_*(X_{l}) - F_{A_l}(X_l)] \mathbb{1}(G_{i - 1}) \mathbb{1}(X_l \in \Unc^{(i - 1)})
\end{equation*}
where the last equality follows because the events $\{X_l \notin \Unc^{(i - 1)}\}$ and $G_{i - 1}$ imply $F_*(X_{l}) = F_{A_l}(X_{l})$.

Now it can be further shown that
\begin{align*}
&[F_*(X_{l}) - F_{A_l}(X_{l})] \mathbb{1}(G_{i - 1}) \mathbb{1}(X_l \in \Unc^{(i - 1)})
\leq \\& [(U^{(i-1)}_1(X_{l}) - L^{(i-1)}_1(X_{l})) + (U^{(i-1)}_0(X_{l}) - L^{(i-1)}_0(X_{l}) )] \mathbb{1}(G_{i - 1}) \mathbb{1}(X_l \in \Unc^{(i - 1)})
\end{align*}
This follows because $\{X_l \in \Unc^{(i - 1)}\}$ implies the two confidence intervals at $X_l$ overlap and $G_{i - 1}$ ensures the true function values $F_0(X_l),F_1(X_l)$ are inside their respective confidence intervals; see the justification after~\eqref{eq:int2} in the full proof.

The main takeaway from the last display is that to bound $\mathbb{E}[ F_*(X_{l}) - F_{A_l}(X_l)]$ it is enough to bound the term 
\begin{equation*}
\mathbb{E} [(U^{(i-1)}_1(X_{l}) - L^{(i-1)}_1(X_{l}))  \mathbb{1}(X_l \in \Unc^{(i - 1)})]
\end{equation*}
as the other term involving $(U^{(i-1)}_0(X_{l}) - L^{(i-1)}_0(X_{l}))$ can be bounded similarly.

\item Step 2: 
Now note that the band functions $U^{(i-1)}_1,L^{(i-1)}_1$ are built out of the data points in $\mathcal{S}^{(i - 1)}_1.$ Throughout this step, we will argue after conditioning on all the events till epoch $i - 2$. Even though the number of points $|\mathcal{S}^{(i - 1)}_1| \sim Bin(N_{i - 1},\Leb(\Unc^{(i - 2)}))$ is random; for this proof sketch we will pretend that $|\mathcal{S}^{(i - 1)}_1|$ is deterministic and equals $N_{i - 1} \Leb(\Unc^{(i - 2)}))$ which is the mean of a $Bin(N_{i - 1},\Leb(\Unc^{(i - 2)}))$ random variable. This is justified in the main proof using an appropriate tail bound for a Binomial random variable and by considering the two cases where $\Leb(\Unc^{(i - 2)})$ is small and large separately. Furthermore, we observe that each of the context points in $\mathcal{S}^{(i - 1)}_1$ is distributed as $\Unif(\Unc^{(i - 2)}).$ 

Now let us write the trivial upper bound
$$\mathbb{E} [(U^{(i-1)}_1(X_{l}) - L^{(i-1)}_1(X_{l}))]  \mathbb{1}(X_l \in \Unc^{(i - 1)}) \leq \mathbb{E} [(U^{(i-1)}_1(X_{l}) - L^{(i-1)}_1(X_{l}))]  \mathbb{1}(X_l \in \Unc^{(i - 2)}).$$

The advantage of writing the above trivial display is that now we recognize that we can upper bound the R.H.S in the last display by using Proposition~\ref{prop:isoband}, which is our main result on the average width of our confidence bands. 

By setting $A = \Unc^{(i - 2)}$,$n = N_{i - 1} \Leb(\Unc^{(i - 2)})$ and $Z = X_l$ in ~\eqref{eq:width} of Proposition~\ref{prop:isoband} all of whose conditions are met, we can deduce that
\begin{align*}
&\mathbb{E} [(U^{(i-1)}_1(X_{l}) - L^{(i-1)}_1(X_{l}))]  \mathbb{1}(X_l \in \Unc^{(i - 1)}) \lesssim \\& \Leb(\Unc^{(i - 2)}) \min\{[N_{i - 1} \Leb(\Unc^{(i - 2)})]^{-1/3},K^{1/2}[N_{i - 1} \Leb(\Unc^{(i - 2)})]^{-1/2}\} \leq \\&
\min\{N_{i - 1}^{-1/3},K^{1/2} N_{i - 1}^{-1/2}\}
\end{align*}
where in the first inequality we have used the $\lesssim$ symbol to ignore logarithmic and constant factors and negligible terms and in the last inequality we have used the trivial bound $\Leb(\Unc^{(i - 2)}) \leq 1.$

\item Step 3: 
We now bring everything together in this step. Since $(X_l,A_l)$ are i.i.d within epoch $i$ we can simply sum the last display in the previous step $N_i$ times to obtain 
\begin{equation*}
R_i = \mathbb{E} \sum_{l = \bar{N}_{i-1}+1 }^{\bar{N}_i}\mathbb{E}[ F_*(X_{l}) - F_{A_l}(X_l)] \lesssim N_i \min\{N_{i - 1}^{-/3},K^{1/2},N_{i - 1}^{-1/2}\} \leq 2  \min\{N_{i - 1}^{2/3},K^{1/2},N_{i - 1}^{1/2}\}
\end{equation*}
where we have used the fact that $N_i \leq 2 N_{i - 1}$ which further implies $N_i = 2^{i - 1} N_{1}$ for $i \in [K].$ 

Now we can sum the last display over all $i \in [K]$ to finally obtain
\begin{equation*}
\mathrm{Regret}_{T}(\pi^*) \lesssim \sum_{i \in [K]} \min\{N_1^{2/3} 2^{2i/3}, K^{1/2}  N_1^{1/2} 2^{i/2} \} \lesssim \min\{N_1^{2/3} 2^{2K/3},K^{1/2}  N_1^{1/2} 2^{K/2}\} \lesssim \min\{T^{2/3},K^{1/2} T^{1/2}\}
\end{equation*}
where the last inequality follows because $K \leq \log_{2}(1 + \frac{T}{N_1}) - 1.$ This finishes our proof sketch. 
\end{itemize}

\section{Confidence Bands for Isotonic Quantile Functions}
\label{sec:confidence_sets} 
In this section we explain our confidence band construction and the associated results. The readers can treat this as an independent section. Consider the sequence model 
\begin{align}
y_i = \theta_i^* + \varepsilon_i, \label{eq:model} 
\end{align} 
where $\varepsilon_i \sim F$ are iid. Throughout, we assume that the sequence $\mathbf{\theta}^* = (\theta_i^*)_{i=1}^{n} \in \mathcal{M}$, where $\mathcal{M}$ is the set of bounded isotonic sequences
\begin{align}
\mathcal{M} = \{ \theta \in \R^{n}: 0 \leq \theta_1 \leq \cdots \leq \theta_n \leq 1\}. \nonumber 
\end{align} 

\begin{remark}
The assumption that $\theta^*$ has entries in $[0,1]$ is important in our analysis. However the specific bounds $0$ and $1$ are arbitrary, and can be easily replaced by any two known real numbers $a$ and $b$.
\end{remark}

\noindent
For any random variable $X$ and $\tau \in (0,1)$, we define $q$ to be a $\tau$-quantile so that 
\begin{align}
\mathbb{P}[X<q] \leq \tau \leq \mathbb{P}[X \leq q]. \nonumber 
\end{align} 
An equivalent way to define a $\tau$-quantile of $X$ is 
\begin{equation*}
q = \argmin_{a \in \R} \E \rho_{\tau}(X - a)
\end{equation*}
where $\rho_{\tau}: \R \rightarrow \R$ is the piecewise linear convex function $x \rightarrow x(\tau - \mathbb{1}(x < 0)).$ For any finite set of numbers $\{z_1, \cdots, z_n\}$, we use $\tau(z_{1:n})$ to denote any $\tau$-quantile of the empirical distribution of this set.

We assume that the error distribution $F$ has a  unique $\tau$ quantile equal to $0$. For notational convenience, we denote this henceforth as $\tau(F) =0$. This implies that the unique $\tau$ quantile sequence of $y$ is $\theta^*.$ Estimation of the true $\tau$ quantile sequence $\theta^*$ based on the sequence $\mathbf{y}$ is a natural problem in this context. 



Having observed the data $\mathbf{y}$ from \eqref{eq:model}, we use the isotonic quantile regression estimator
\begin{align}
\hat{\theta} = \textrm{argmin}_{\theta \in \mathcal{M}} \sum_{i=1}^{n} \rho_{\tau}(y_i - \theta_i). \label{eq:estimator} 
\end{align}

The estimator $\hat{\theta}$ need not be uniquely defined because of the lack of strong convexity of the $\rho_{\tau}$ function. When there are multiple solutions, $\hat{\theta}$ can be defined by choosing any one of the solutions according to some predefined deterministic rule. We will now construct confidence bands for $\theta^*$ based on the estimator $\hat{\theta}$.

\subsection{Band Construction} 
\label{subsec:band_construction} 
We describe our band construction procedure in detail in this section. 
Our algorithm takes as input two constants $\Gamma_1,\Gamma_2$ and the data vector $\mathbf{y} \in \R^n$ and outputs two sequences $\hat{\theta}^{u},\hat{\theta}^{l}.$ The sequences $\hat{\theta}^{u},\hat{\theta}^{l}$ correspond to the upper/lower confidence bands for the true sequence $\theta^*.$  We refer to this algorithm as the \textit{BAND($\Gamma_1$,$\Gamma_2$)} algorithm.

\begin{itemize}
	\item[(i)] Compute the isotonic $\tau$ quantile estimator $\hat{\theta}$ defined in \eqref{eq:estimator}. 
	\item[(ii)] For $1\leq i \leq n$, let $\hat{l}_i = \min\{k: \hat{\theta}_k = \hat{\theta}_i\}$ and $\hat{u}_i = \max\{ k : \hat{\theta}_k = \hat{\theta}_i\}$. In words, $\hat{l}_i,\hat{u}_i$ are the left and right endpoints of the constant piece of $\hat{\theta}$ containing $i.$
	\item[(iii)] Define the random set $\hat{G} \subset [n]$ as follows:
	\begin{equation*}
	\hat{G} = \{i \in [n]: \min\{\hat{u}_i - i + 1, i - \hat{l}_i + 1\} \geq \Gamma_2 \log n\}.
	\end{equation*}

	\item[(iv)] Take any $i \in \hat{G}.$ Define 
	\begin{equation*}
	\hat{\theta}^{(u)}_i = \min \Big\{\hat{\theta}_i + \frac{\Gamma_1 \sqrt{\log n}}{\sqrt{\hat{u}_i - i + 1,}},1 \Big\}.
	\end{equation*}
	Also define 
	\begin{equation*}
	\hat{\theta}^{(l)}_i = \max\Big\{\hat{\theta}_i - \frac{\Gamma_1 \sqrt{\log n}}{ \sqrt{i - \hat{l}_i + 1}},0 \Big \}.
	\end{equation*}
	This completes the definition of $\hat{\theta}^{(u)},\hat{\theta}^{(l)}$ on $\hat{G}.$
	


\item[(v)] Now take any $i \in \hat{G}^{c} = [n] - \hat{G}.$ Define the integer 
$\hat{k}_{u}(i) = \min\{j \in [n]: j \in \hat{G}, j > i\}$ which is well defined if the set $\{j \in \hat{G}, j > i\}$ is non empty. By definition, $\hat{k}_{u}(i) \in \hat{G}.$
In this case, define $$\hat{\theta}^{(u)}_i = \hat{\theta}^{(u)}_{\hat{k}_{u}(i)}.$$
Otherwise, if the set $\{j \in \hat{G}, j > i\}$ is empty, define $\hat{\theta}^{(u)}_i = 1.$

Similarly, define the integer 
$\hat{k}_{l}(i) = \max\{j \in [n]: j \in \hat{G}, j < i\}$ which is well defined if the set $\{j \in \hat{G}, j < i\}$ is non empty. Again note that $\hat{k}_{l}(i) \in \hat{G}.$
In this case, define $$\hat{\theta}^{(l)}_i = \hat{\theta}^{(l)}_{\hat{k}_{l}(i)}.$$
Otherwise, if the set $\{j \in \hat{G}, j < i\}$ is empty, define $\hat{\theta}^{(l)}_i = 0.$

\item[(vi)]
The band sequences $\hat{\theta}^{(l)},\hat{\theta}^{(u)}$ defined so far need not be monotonic. We perform a final post-processing step to ensure that the constructed bands are monotonic. Specifically, for the upper confidence band, we perform a backward pass through $\hat{\theta}^{(u)}$ and set $\hat{\theta}^{(u)}_i = \hat{\theta}^{(u)}_{i+1}$ if $\hat{\theta}^{(u)}_i > \hat{\theta}^{(u)}_{i+1}$. The operation on the lower bound is similar. Note that this final step conserves the coverage properties of the constructed set, and cannot increase the average width of the confidence band. 
 \end{itemize}

\begin{remark}\label{rem:comp}
The estimator $\hat{\theta}$ can be seen as a solution of a linear program and hence can be computed efficiently. In fact, it can be computed in near linear $O(n \log n)$ time as shown in~\cite{hochbaum2017faster}. Since $\hat{\theta}$ can be computed in near linear time and all the operations decribed after (i) can be performed in $O(n)$ time, it is clear from the above construction that the band functions $\hat{\theta}^{(l)},\hat{\theta}^{(l)}$ can also be computed in $O(n \log n)$ time. 
\end{remark}

\subsection{Validity and Width of the Confidence Bands} 


To state our formal guarantees on the coverage properties of the confidence band introduced in the last section, we need a ``local growth" assumption on the error distribution $F$. 
This is a standard assumption in the quantile regression literature. 

\noindent 
\textbf{\textit{Assumption A:}} There exist constants $L, \tilde{C} > 0$ such that 
\begin{equation}\label{eq:assum}
|F(t) - F(0)| > \tilde{C} t \:\:\text{for}\:\: t \in [-L,L].
\end{equation} 

\begin{remark}
\label{rem:growth}
Assumption $A$ is implied by a standard assumption in the quantile regression literature which assumes that the density (w.r.t lebesgue measure) of the errors is lower bounded by a positive number in a neighborhood of the relevant quantile; see condition 2 in \cite{he1994convergence} and condition D.1 in \cite{belloni2011}. The above assumption ensures that the quantile of $y_i$
is uniquely defined and there is a
uniformly linear growth of the CDF around a neighbourhood of the quantile. An assumption of such a flavor
(making the quantile uniquely defined) is clearly going to be necessary. We think this is a mild assumption
on the distribution of $y_i$ as this should hold for most realistic error distributions. For example, if the errors are i.i.d draws from any density with respect to the Lebesgue measure that is bounded away
from zero on any compact interval then our assumption will hold. In particular, no moment assumptions are
being made on the distribution of the errors. 
\end{remark}
\noindent 
The next result establishes that upon suitable choice of the parameters $\Gamma_1$, $\Gamma_2$, the constructed bands enjoy finite sample coverage property. 

\begin{theorem}\label{thm:valid}
	\label{lem:coverage} 
	Suppose Assumption A holds. Fix a confidence level $0 \leq \alpha \leq 1.$ $\Gamma_1, \Gamma_2$ are inputs to the \textit{BAND($\Gamma_1$, $\Gamma_2$)} algorithm chosen in the following way:
	Choose $\Gamma_1>0$ large enough satisfying
	\begin{equation}\label{eq:cond1}
	2 \log 3 \:\:(\tilde{C}^2 \Gamma_1^2 - 1) \geq \log \frac{1}{\alpha}
	\end{equation}
	and then choose $\Gamma_2>0$ large enough satisfying
	\begin{equation}\label{eq:cond2}
	\frac{\Gamma_1}{\sqrt{\Gamma_2}} \leq L.
	\end{equation}
	Then the following finite sample coverage guarantee holds for the output of the \textit{BAND($\Gamma_1$,$\Gamma_2$)} algorithm $\hat{\theta}^l,\hat{\theta}^u \in \R^n$,
	\begin{align}
	\mathbb{P}\Big[ \hat{\theta}_i^l \leq \theta_i^* \leq \hat{\theta}_i^u\:\:\forall i \in [n] \Big] \geq 1 - \alpha. \nonumber 
	\end{align}
\end{theorem}

\noindent 
Our next theorem gives an adaptive bound to the average width of our confidence band.
\begin{theorem}\label{thm:width}
	 There exists a universal constant $C > 0$ (independent of the problem parameters) such that the outputs 
	$\hat{\theta}^l,\hat{\theta}^u \in \R^n$ satisfy
	\begin{equation*}
	\max\Big\{\mathbb{E} \sum_{i = 1}^{n} (\hat{\theta}^{(u)}_j - \hat{\theta}_j), \mathbb{E} \sum_{i = 1}^{n} (\hat{\theta}_j - \hat{\theta}^{(l)}_j)\Big\} \leq C \left(\Gamma_1 \min\{n^{2/3} \sqrt{\log n},  \sqrt{kn} \log n\} + \Gamma_2 \log n \min\{n^{1/3}, k \log n\}\right),
	\end{equation*}
	where  $k$ denotes the number of constant pieces of the true quantile sequence $\theta^*$.
\end{theorem}

\begin{remark}
It is easy to see that the above result implies a bound scaling like $\tilde{O}(\min\{n^{-1/3},k^{1/2} n^{-1/2}\})$ on $\frac{1}{n} \mathbb{E} \sum_{i = 1}^{n} (\hat{\theta}^{(u)}_j - \hat{\theta}^{(l)}_j)$ where we can interpret $\frac{1}{n} \mathbb{E} \sum_{i = 1}^{n} (\hat{\theta}^{(u)}_j - \hat{\theta}^{(l)}_j)$ as the average width of our confidence band. Such an automatic adaptive inference result appears to be new. The above inference result also implies adaptive estimation rates for the Isotonic Quantile Regression (IQR) estimator which was also not established before to the best of our knowledge.
\end{remark}
\subsection{Confidence bands: discrete to continuum}
We extend our confidence bands from the sequence model to the random design setting.  
Fix any positive integer $n$ and $0 < \alpha < 1.$ Given any sequence of data points $\{x_i,y_i\}_{i = 1}^{n} \in ([0,1] \times \R)^n$ we will now define two functions $U^{\alpha}_{\{x_i,y_i\}_{i = 1}^{n}}: [0,1] \rightarrow \R$ and $L^{\alpha}_{\{x_i,y_i\}_{i = 1}^{n}}: [0,1] \rightarrow \R$ which will play the role of $1 - \alpha$ confidence bands for an underlying non decreasing regression function. To avoid notational burden we will often drop the superscript and the subscripts and just refer to the functions as $U,L.$ It is to be understood from the context that they are a function of the set of data points and $\alpha.$


Let $x_{(j)}$ be the $j$th order statistic of $\{x_i\}_{i = 1}^{n}.$ Let $y_{(j)}$ denote the entry of $\{y_i\}_{i = 1}^{n}$ which corresponds to $x_{(j)}.$ Let $v \in \R^n$ such that $$v_{j} = y_{(j)}.$$ Now let $u,l \in \R^n$ be the outputs of the BAND($\Gamma_1$,$\Gamma_2$) algorithm applied to the input vector $v$ where $\Gamma_1$,$ \Gamma_2$ satisfy~\eqref{eq:cond1} and~\eqref{eq:cond2}.



For any $j \in [n]$, define 
\begin{align*}\label{eq:band}
U(x_{(j)}) = u_j, \,\,\,\,
L(x_{(j)}) = l_j.
\end{align*} 

\noindent 
This defines $U,L$ at the design points $\{x_i\}_{i = 1}^{n}.$ To define $U,L$ at other points we simply do a piece-wise constant interpolation as follows. 
For a general $x \in [0,1]$, define 
$x^{+} = \{\min x_i: x_i \geq x\}$.  If the defining set above is empty, define $x^{+} = 1.$
Similarly, define 
$x^{-} = \{\max x_i: x_i \leq x\}.$
If the defining set above is empty, define $x^{-} = 0.$
Now for any $x \in [0,1]$, define $U(x) = U(x^{+})$ and $L(x) = L(x^{-})$. 
Note that by definition, $U(x) \geq L(x)$ for all $x \in [0,1]$ and $U(x),L(x)$ are both non decreasing functions. 


\begin{proposition}(Coverage and Average Width of  Confidence Band)\label{prop:isoband}
Let $A \subset [0,1]$ be a finite disjoint union of intervals and $n$ be any positive integer. Let $x_1,\dots,x_n$ be $n$ i.i.d samples from $\Unif(A).$ Let $y_i = f(x_i) + \epsilon_i$, $f \in \mathcal{F}_{[0,1]\uparrow}$ and $\varepsilon_i \sim F$ are iid. 
Suppose the error distribution $F$ satisfies $\tau(F) = 0$ and Assumption $A$ \eqref{eq:assum}. Let $U \colon= U^{\alpha}_{\{x_i,y_i\}_{i = 1}^{n}}$ and $L \colon=L^{\alpha}_{\{x_i,y_i\}_{i = 1}^{n}}$ denote the corresponding confidence band functions. Then the following confidence statement holds:
	
	\begin{equation}\label{eq:confi}
	\mathbb{P}\big(L(x) \leq f(x) \leq U(x) \:\:\forall x \in [0,1]\big) \geq 1 - \alpha.
	\end{equation}
	
	\noindent
	Moreover, let $Z \sim \Unif(0,1)$ be independent of ${\{x_i,y_i\}_{i = 1}^{n}}.$ Then the following average confidence width bound holds:
	\begin{equation}\label{eq:width}
	\mathbb{E} \big[(U(Z) - L(Z)) \mathbb{I}(Z \in A)\big] \leq \Leb(A) \big[\log n \cdot  b_n + r_n].
	\end{equation} 
	Here $b_n$ (defined below) is the same as what appears in the bound (R.H.S  in Theorem~\ref{thm:width}) (except that the precise value of $C$ may be changed) on the average width of the confidence band in the sequence model setting.
	$$b_n = C \left(\Gamma_1 \min\{n^{2/3} \sqrt{\log n},  \sqrt{kn} \log n\} + \Gamma_2 \log n \min\{n^{1/3}, k \log n\}\right).$$
	In the display above, $k$ is the cardinality of the set $\{f(x): x \in [0,1]\}$ which is either a positive integer or $+\infty$. The additional term $r_n$ in~\eqref{eq:width} is the following:
	$$r_n = C' \Big( \frac{\log n}{n} + \frac{1}{n^2} \Big) ,$$
	where $C'$ is another universal constant independent of the problem parameters. 
\end{proposition}

\begin{remark}
Proposition~\ref{prop:isoband} essentially guarantees a similar bound on the average width of our confidence band in the random design setting as that guaranteed by Theorem~\ref{thm:width} for the sequence model. This follows as $r_n$ is a lower order term compared to $b_n$.
\end{remark}

\noindent 
We collect several remarks explaining some aspects of our confidence band results.
\begin{remark}
The coverage of our confidence band relies on  Hoeffding's inequality, and is therefore valid for any finite sample size. 
\end{remark}

\begin{remark}
The width of our band is rate optimal since the expected width scales like $\tilde{O}(n^{-1/3})$ which is the optimal rate of estimation for isotonic functions. Moreover, the width of our band is adaptive to piecewise constant structure. If the true $\tau$ quantile function $f$ is piecewise constant with $k$ pieces, the width of our band scales like $\tilde{O}(\sqrt{k/n}).$ Note that this adaptivity is automatic in the sense that there is no explicit regularization necessary to achieve the adaptive rate. 
\end{remark}

\begin{remark}
For the above theorem, we only require a mild restriction on $F$; namely it satisfies assumption A in~\eqref{eq:assum}. In particular, no moment assumptions are required and Theorem~\ref{thm:widthinform} holds even when $F$ is the Cauchy distribution. 
\end{remark}

\begin{remark}
To set $\Gamma_1,\Gamma_2$ satisfying~\eqref{eq:cond1} and~\eqref{eq:cond2} our confidence band algorithm needs to know local growth parameters $\tilde{C},L$ (as in~\eqref{eq:assum}) of the underlying error distribution $F.$ Of course, in practice we would not know the error distribution. However, setting valid values of $\tilde{C},L$ so that the error distribution $F$ satisfies assumption A should not be a problem for most realistic distributions $F.$ We can safely set $L = 1$ and a small enough value of $\tilde{C}.$ 
\end{remark}

\section{Discussions} 
\label{sec:discussions} 
Here we discuss a couple of naturally related matters.
\begin{itemize}
    \item[(i)] Our confidence band method is based on the Isotonic Quantile Estimator.  One can alternatively use the Isotonic LSE, and construct an analogous confidence band. This construction would then be similar to the one  proposed in~\cite{yang2019contraction}. We claim that our proof technique can be extended to prove validity and average width guarantee for this band (based on LSE) as well. We do not include the proof here because of space considerations. This would provide a very different proof of the validity of the band compared to that in ~\cite{yang2019contraction}. The proof there uses a specific property of the Isotonic LSE called the NUNA (non increasing under neighbor averaging). This idea appears to be hard to generalize to the quantile setting. On the other hand, we find that our proof technique can be used both for the Isotonic LSE and the quantile estimator. It is also worth reiterating here that our adaptive bound guarantee on the average width appears to be the first one explicitly stated and proved for a isotonic or any other shape constrained confidence band.

    The main reason for us analyzing the Isotonic Quantile Estimator is that we want our confidence band to be robust to heavy tailed errors. For errors with no moments like the Cauchy distribution, Isotonic LSE performs poorly in practice, see Figure~\ref{fig:lsevsmedian}.  
    
    \item[(ii)] Theorem~\ref{thm:width} gives an adaptive bound on the average width of our confidence band in the sequence model. The main ingredient in the proof of Theorem~\ref{thm:width} is Proposition~\ref{prop:pieces} which bounds the average number of constant pieces of the Isotonic Quantile Estimator. This is one of the main technical contributions of this paper. A corresponding bound for the average number of constant pieces of the Isotonic LSE (when the errors are gaussian) is known; see Theorem $1$ in~\cite{MW00} and Theorem $4.2$ in~\cite{bellec2016adaptive}. However, controlling the number of constant pieces for the Isotonic Quantile Estimator has not been attempted before and our bound is potentially of independent interest. At a high level, our proof technique adapts the general proof strategy of~\cite{MW00} to the non gaussian and quantile setting. Essentially, we reduce the problem of bounding the number of pieces of the Isotonic Quantile Estimator to bounding the maxima of certain random walks with $\pm 1$ increments. We then prove the required bounds on the maxima of these random walks.

    \item[(iii)] A natural question that arises as a follow up to our work is whether similar adaptive confidence bands and consequently policies with adaptive regret bounds can be constructed for multivariate isotonic regression and other shape constraints like convexity. The extension to higher dimensions for isotonic regression appears to not be straightforward. Constructing finite sample valid adaptive confidence band methodology for other shape constraints like convexity also appears to be an open problem. We leave these intriguing follow up questions for future research.

    \item[(iv)] It is known in the Multi Armed Bandit literature that certain algorithms adapt to the gap between the means of the two arms. This gives instance dependent bounds and leads to logarithmic regret in easy problems when the gap between the means of the two arms is at least a constant; see~\cite{foster2020instance} and references therein. One can ask the question whether the same phenomenon is possible for the Isotonic Contextual Bandits problem initiated in this paper and whether the policy proposed in this paper can exhibit logarithmic regret for certain easy problems. We again leave this question for future research.

    
    

\end{itemize}

\section{Simulations}\label{sec:sims}
In this section, we collect some numerical experiments exploring the finite sample performance of our confidence bands. Our simulations are conducted in \texttt{R}. We fit the isotonic quantile regression functions using the R package \texttt{isotone}. 

\begin{enumerate}
    \item Figure~\ref{fig:gaussian_strict}: Consider $f(x) = x$ and $\theta^* = f(i/n)$ for $i \in 1,\dots,n$. We generate a dataset with $n=500$, and additive iid  Gaussian errors with mean zero and standard deviation $0.1$. We plot the data set, along with the isotonic median fit (in \texttt{RED}) in the right pane. The left pane plots the true median sequence in \texttt{BLACK}, and the confidence bands in \texttt{RED}. For these simulations, we use $\Gamma_1 = 0.5$, $\Gamma_2= 0.5$.


    \item Figure~\ref{fig:gaussian_piecewise}: We consider the same setting as in Figure~\ref{fig:gaussian_strict}, except that we change the function $f(x)$ to a piecewise constant function $f(x) = 0.1 + 0.2 \cdot \lfloor 5x \rfloor$. 
    
    \item Figure~\ref{fig:cauchy}: This figure is generated using the same median sequence as  Figure~\ref{fig:gaussian_strict}, but with additive Cauchy noise. We use a Cauchy distribution with $\mathrm{location}=0$ and $\mathrm{scale}=0.1$. We use $\Gamma_1 = 0.5$, $\Gamma_2=0.5$ to construct the confidence bands. In the left pane, we see that there are several data points which have large magnitude compared to the signal value. The signal takes values between $0$ and $1$, whereas several points are larger than $5$ in absolute value. In fact, in this instance there are a couple of points with magnitude more than $100$, which we deleted to achieve better visualization. In this heavy tailed case, our confidence band methodology gives a reasonable solution as shown in the right pane. 
    
    \item Figure~\ref{fig:lsevsmedian}: In the same setting as Figure~\ref{fig:cauchy}, we plot the isotonic median fit (in \texttt{RED}) and the isotonic least squares fit (in \texttt{BLUE}). We see that the isotonic least squares estimator performs poorly in this setting. This is an instance of the well known fact that for heavy tailed errors, least squares can perform poorly and the quantile estimators are significantly more robust. This motivates our confidence set construction using isotonic quantiles rather than the isotonic median.

    
    \item Figure~\ref{fig:cauchy_7}: We replicate a setting analogous to Figure~\ref{fig:cauchy}, but for the $0.7$-quantile. For the left pane, the data is generated as $y_i = \theta_i^* + \varepsilon_i$, where $\theta_i^* = f(i/n) $, and $f(x)= 0.1 + 0.8 \cdot x$. Thus the population $0.7$-quantile sequence is $\theta_i^* + q_{0.7}$,  where $q_{0.7}$ is the $0.7$-quantile of a Cauchy distribution with location zero and scale $0.1$. The pane on the right follows the same recipe, but uses a piecewise constant function $f(x) = 0.1 + 0.2 \cdot \lfloor 5 x \rfloor$. We use $\Gamma_1=1$, $\Gamma_2=0.75$. 
    
\end{enumerate}

\begin{figure}
\centering 
\begin{subfigure}[b]{0.45\linewidth}
    \centering
    \includegraphics[width= \linewidth]{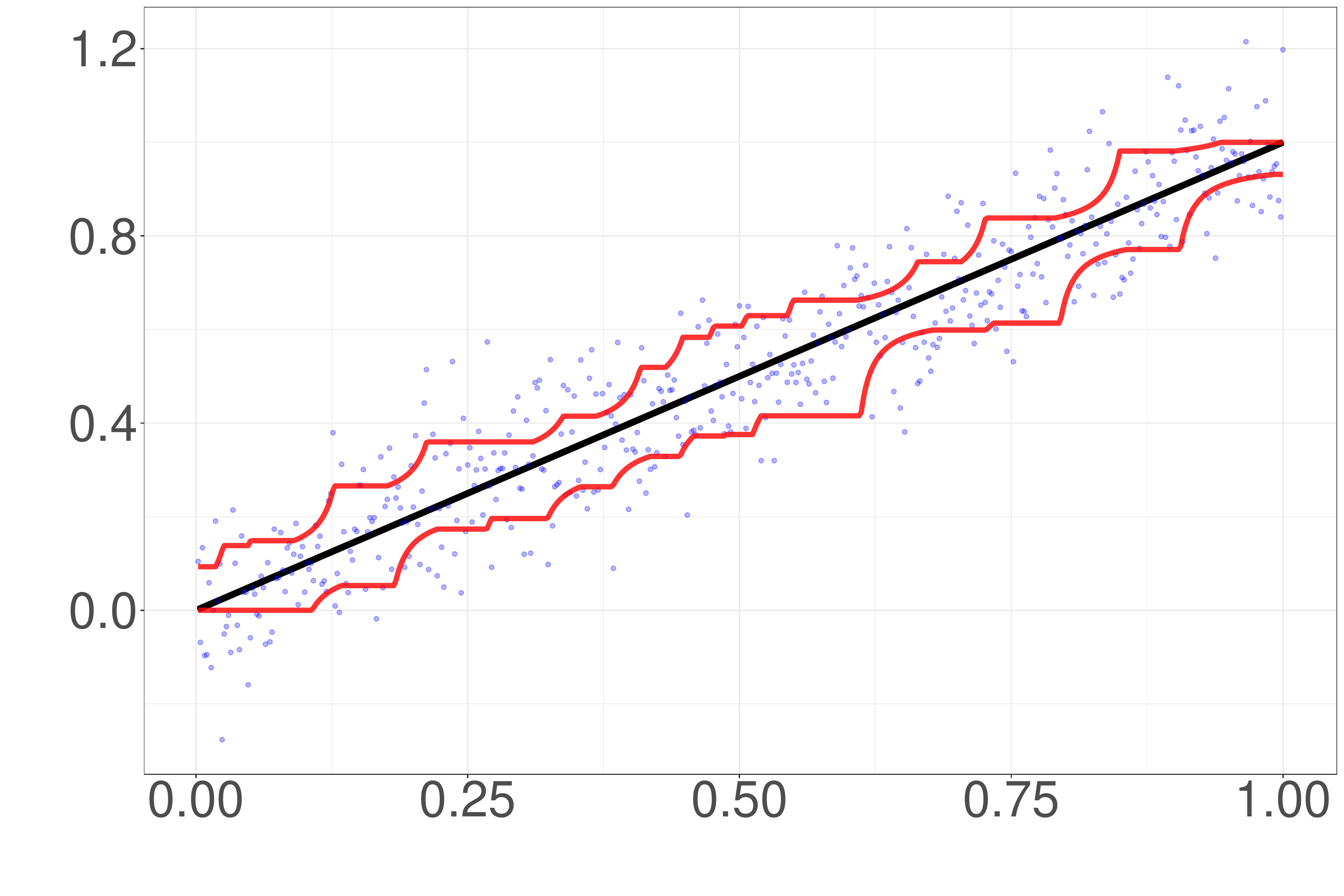}
\end{subfigure}
\begin{subfigure}[b]{0.45\linewidth}
    \centering
    \includegraphics[width=\linewidth]{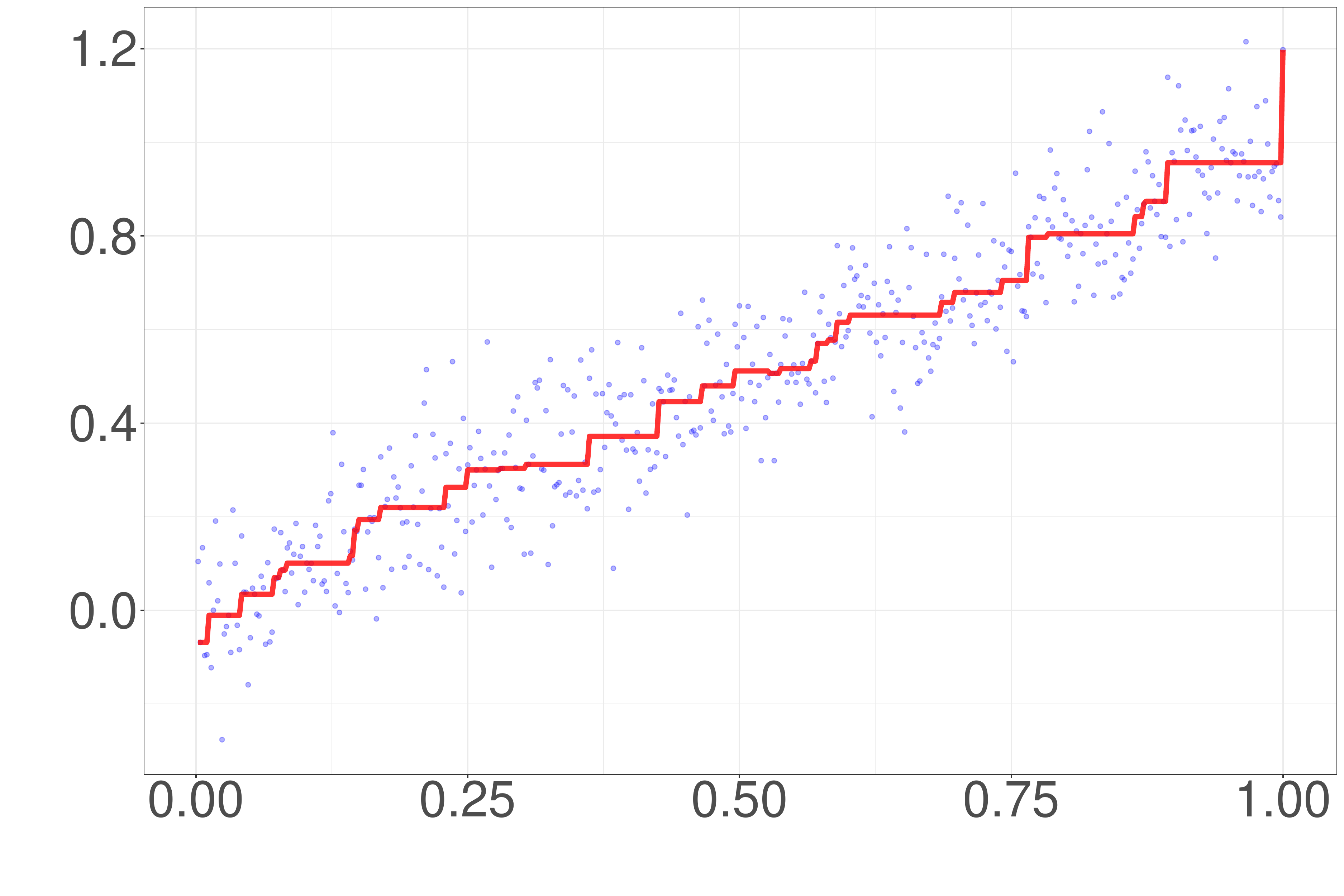}
\end{subfigure}
    \caption{Gaussian data ($\sigma=0.1$), $f(x)=x$, $n=500$. The left pane plots the true median sequence $\theta^*$ (in \texttt{BLACK}) and the confidence bands in \texttt{RED}. The right pane plots the isotonic median fit in \texttt{RED}.}
    \label{fig:gaussian_strict}
\end{figure}

\begin{figure}
\centering 
\begin{subfigure}[b]{0.45\linewidth}
    \centering
    \includegraphics[width= \linewidth]{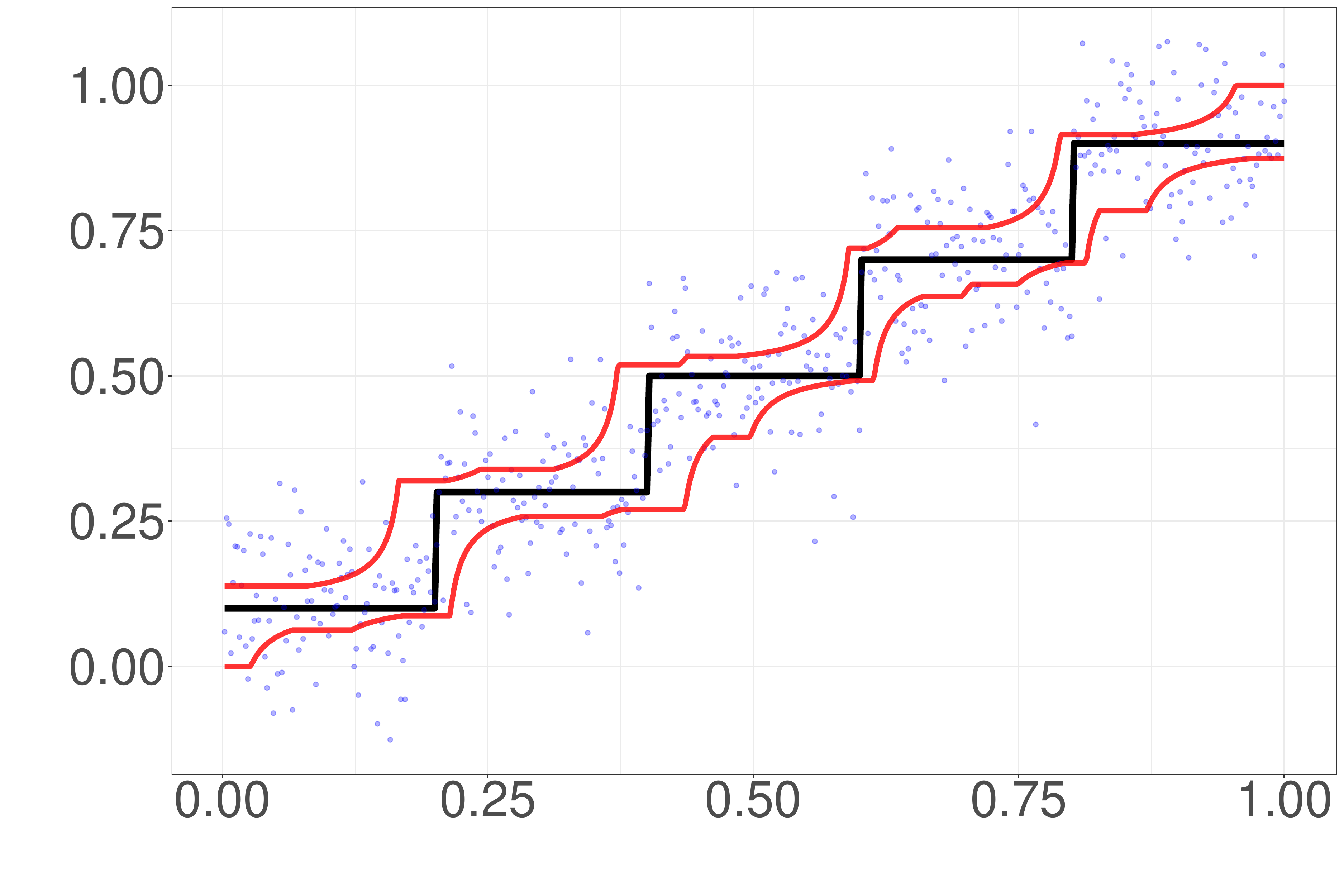}
\end{subfigure}
\begin{subfigure}[b]{0.45\linewidth}
    \centering
    \includegraphics[width=\linewidth]{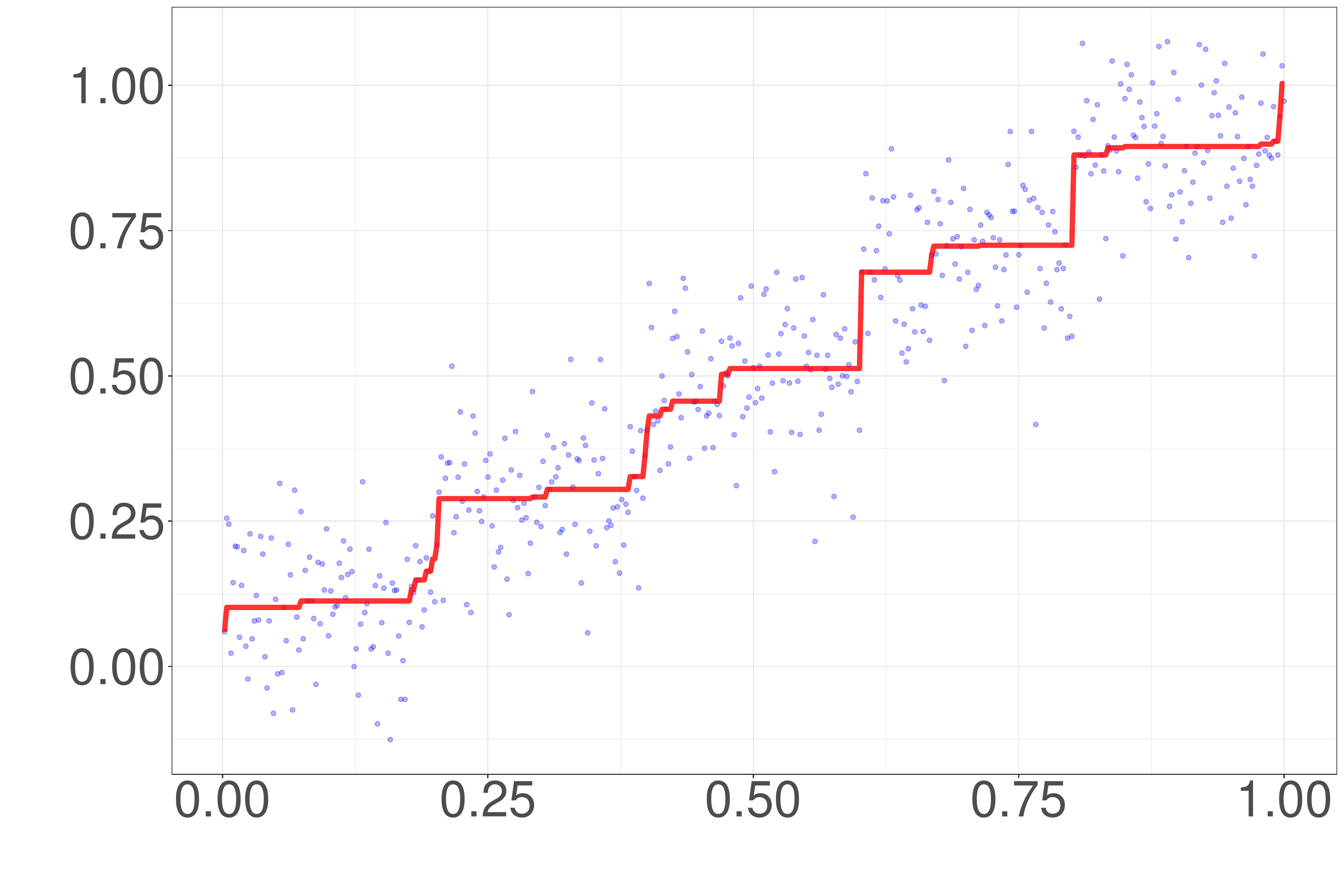}
\end{subfigure}
    \caption{Gaussian data ($\sigma=0.1$), piecewise constant function $f(x) = 0.1 + 0.2 \cdot \lfloor 5x \rfloor$, $n=500$. The left pane plots the true median sequence $\theta^*$ (in \texttt{BLACK}) and the confidence bands in \texttt{RED}. The right pane plots the isotonic median fit in \texttt{RED}.}
    \label{fig:gaussian_piecewise}
\end{figure}

\begin{figure}
\centering 
\begin{subfigure}[b]{0.45\linewidth}
    \centering
    \includegraphics[width= \linewidth]{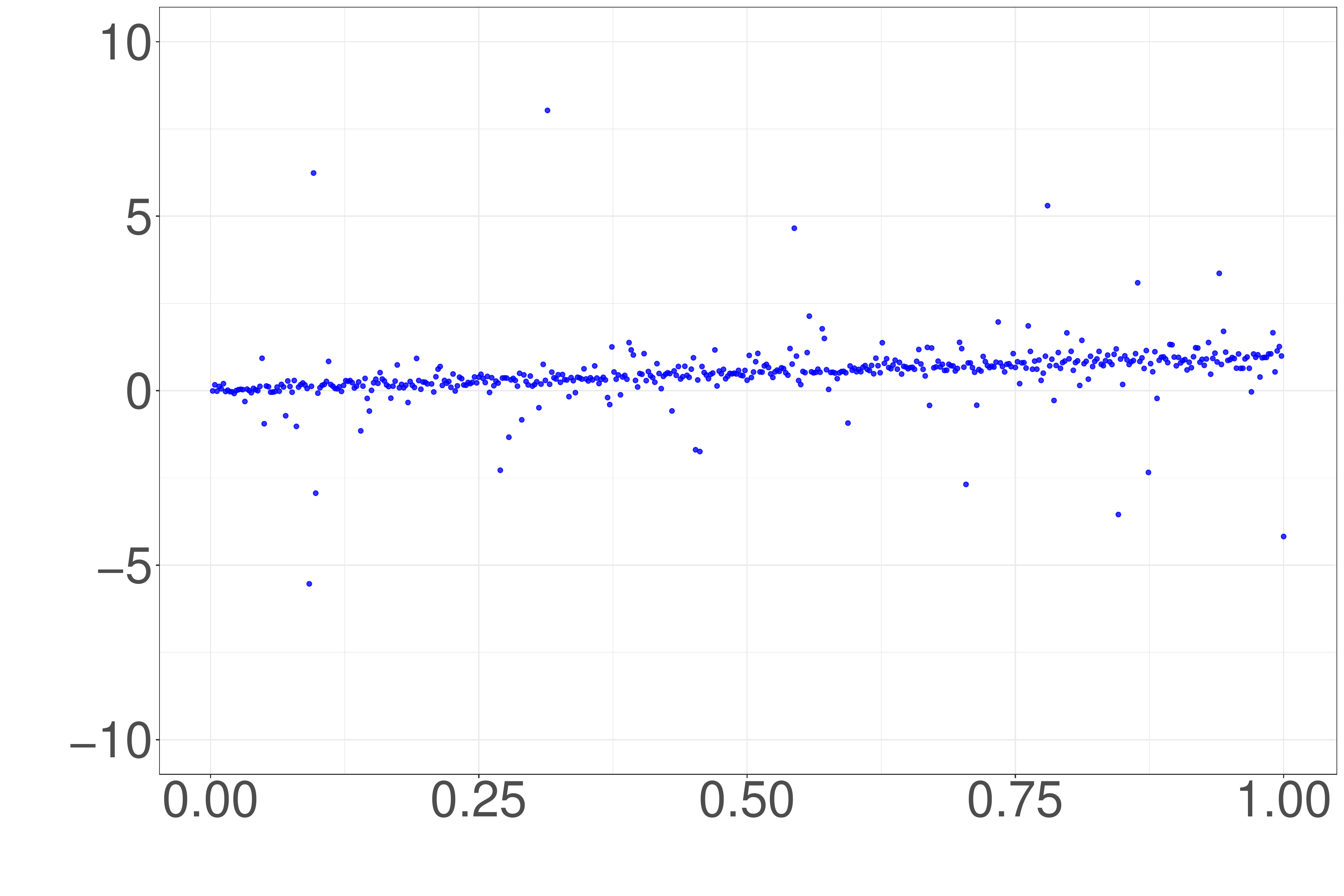}
\end{subfigure}
\begin{subfigure}[b]{0.45\linewidth}
    \centering
    \includegraphics[width=\linewidth]{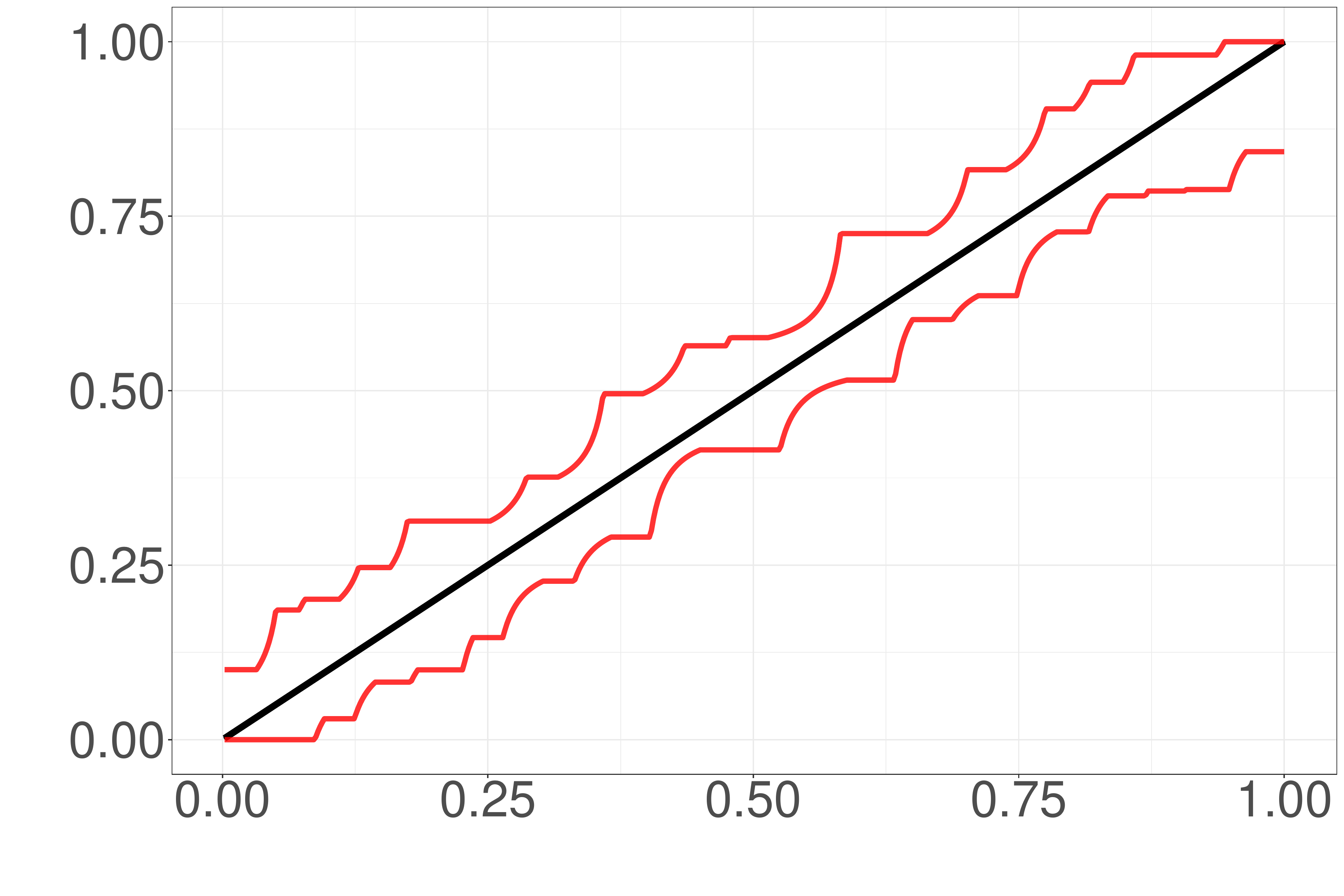}
\end{subfigure}
    \caption{Cauchy data truncated between $[-10,10]$, $\mathrm{location}=0, \mathrm{scale}=0.1$. $f(x)=x$, $n=500$. The left pane plots the truncated data. The right pane plots the true median sequence (in \texttt{BLACK}) and the confidence bands in \texttt{RED}. }
    \label{fig:cauchy}
\end{figure}

\begin{figure}
    \centering
    \includegraphics[scale=0.25]{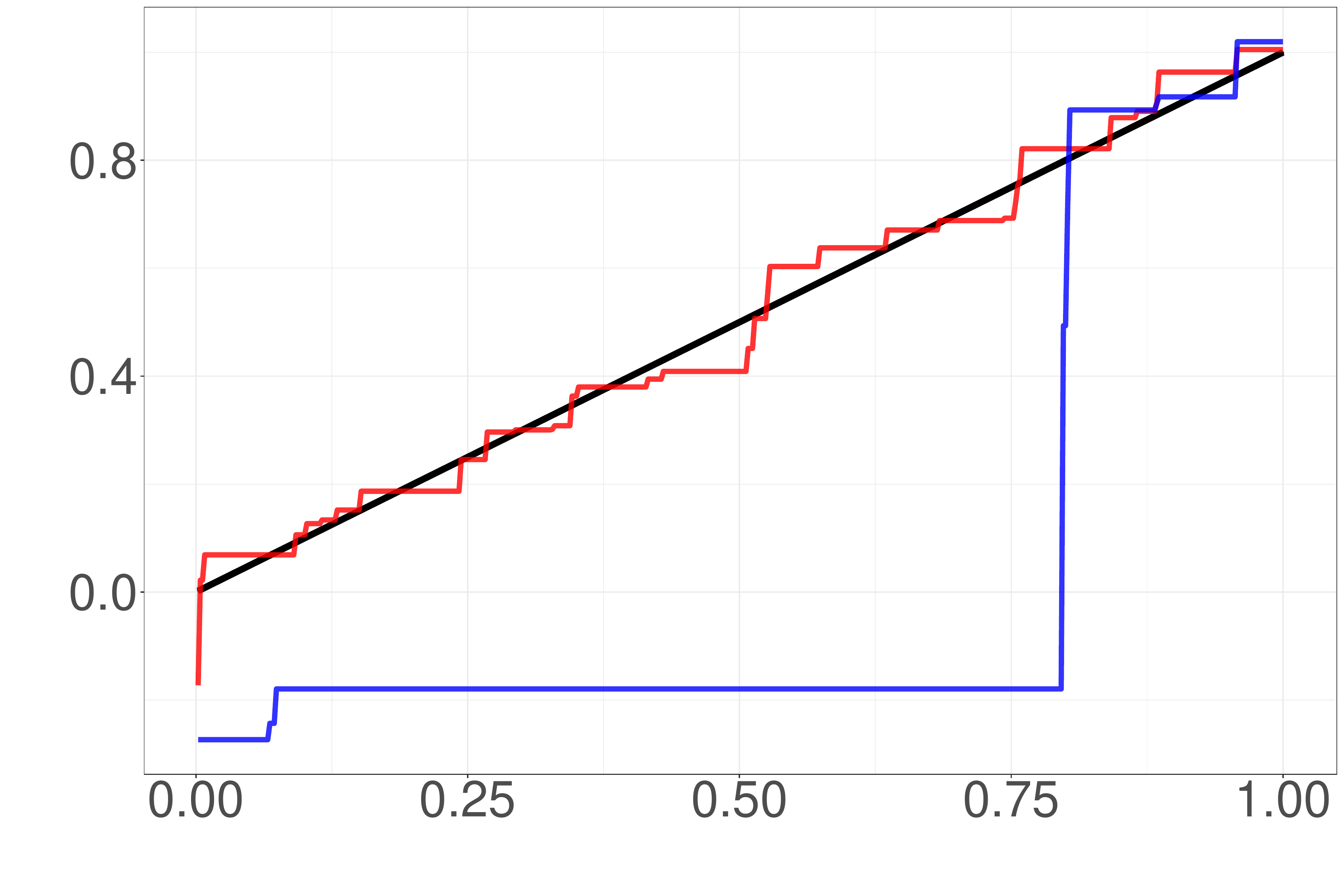}
    \caption{LSE vs. isotonic median, Cauchy data, $\mathrm{location}=0$, $\mathrm{scale}=0.1$, $n=500$. The \texttt{BLACK} line represents the true median sequence. We plot the isotonic median estimator in \texttt{RED}, while   the isotonic Least squares estimator is represented in \texttt{BLUE}.}
    \label{fig:lsevsmedian}
\end{figure}

\begin{figure}
\centering 
\begin{subfigure}[b]{0.45\linewidth}
    \centering
    \includegraphics[width= \linewidth]{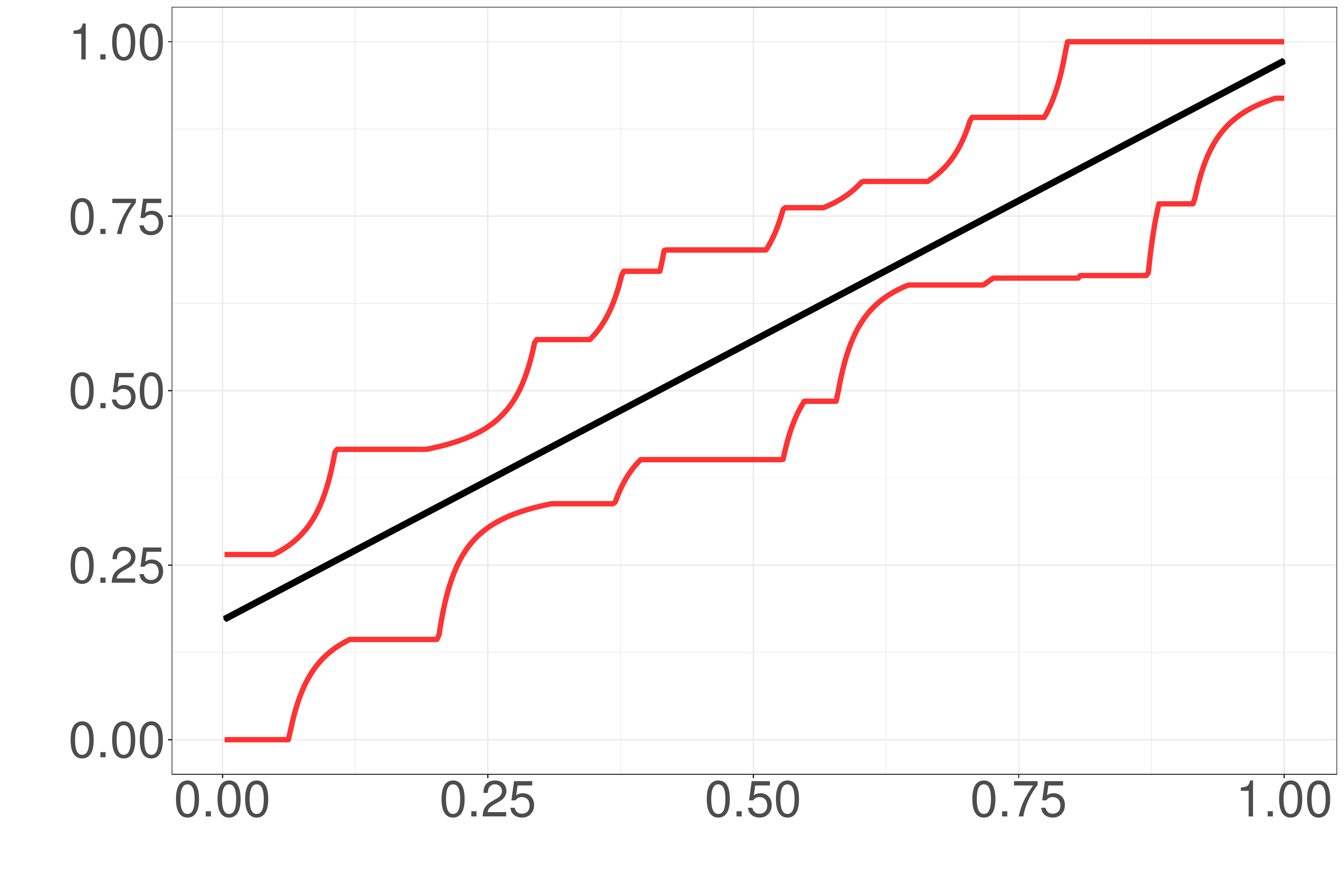}
\end{subfigure}
\begin{subfigure}[b]{0.45\linewidth}
    \centering
    \includegraphics[width=\linewidth]{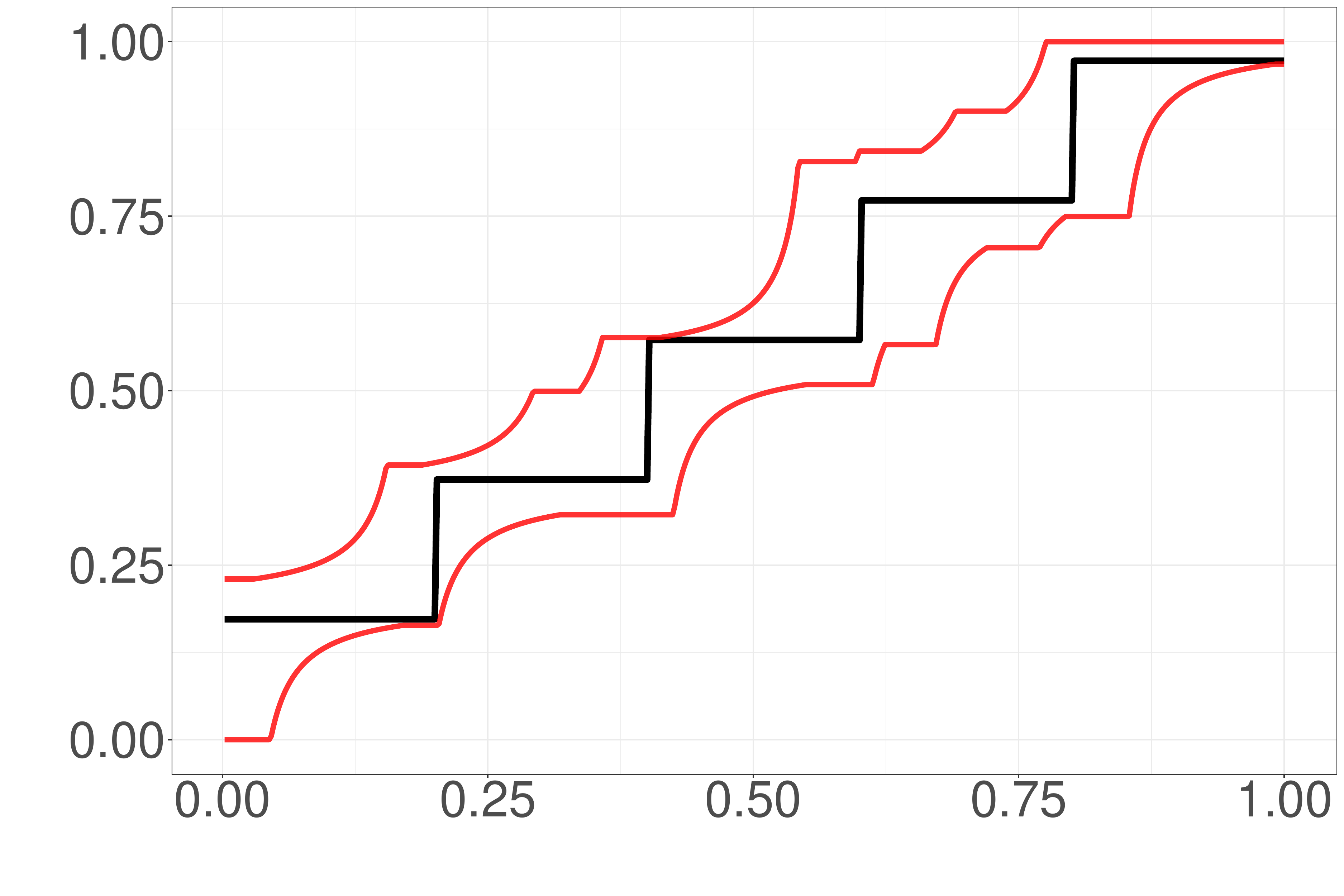}
\end{subfigure}
    \caption{Cauchy data, $\mathrm{location}=0$, $\mathrm{scale}=0.1$, $n=500$. The black curves represent the true $0.7$ quantile sequence (in \texttt{BLACK}) and the confidence bands are visualized in \texttt{RED}.  For the left pane, the true quantile sequence $\theta^*_i = f(i/n) + q_{0.7}$, where $f(x) = 0.1 + 0.8 \cdot x$ and $q_{0.7}$ is the $0.7$-quantile of a Cauchy distribution with location zero and scale $0.1$. For the right pane, $\theta^*_i = f(i/n) + q_{0.7}$, where $f(x) = 0.1 + 0.2 \cdot \lfloor 5 x \rfloor$.   }
    \label{fig:cauchy_7}
\end{figure}

\section{Proofs} 
\label{sec:proofs} 

\subsection{Regret bounds} 
\label{sec:upper_bound_proof} 

\subsubsection{Proof of Theorem~\ref{thm:regret}}
\begin{proof}[Proof of Theorem~\ref{thm:regret}]
We analyze the worst case regret of our policy. To this end,  for $i \in [K]$, define the ``good" event 
\begin{align}
G_i =  \cap_{j = 0}^{1}\{U_{j}^{(i)}(x) \geq F_{j}(x) \geq L_{j}^{(i)}(x)\:\:\forall x \in [0,1]\}. \label{eq:good_event} 
\end{align} 
Further, define the stopping time 
\begin{align}
 \tau = \min \Big\{i \in [K]: \Leb(\Unc^{(i - 1)}) <  \frac{\log T}{\sqrt{T}} \Big\} \label{eq:stopping_time}
\end{align} 
with the convention that the minimum over an empty set is $\infty$. 
Armed with this notation, we can decompose the regret as follows: 
\begin{align}
\mathrm{Regret}_{T}(\pi) &= \sum_{l = 1}^{T} \mathbb{E}[ F_*(X_{l}) - F_{A_l} (X_{l})]. \nonumber \\
&= \sum_{i \in [K]} \sum_{l = \bar{N}_{i-1}+1 }^{\bar{N}_i}\mathbb{E}[ F_*(X_{l}) - F_{A_l}(X_l)] \nonumber \\
&:= \sum_{i \in [K]} \mathbb{E}\Big[ \mathbb{1}(G_{i-1}, i < \tau) \sum_{l = \bar{N}_{i-1}+1 }^{\bar{N}_i} [ F_*(X_{l}) - F_{A_l}(X_l)] \Big] + T_2 \nonumber \\
&:= \sum_{i \in [K]} \mathbb{E}[R_i] + T_2. \label{eq:regret_decomp1} 
\end{align} 
The first term dominates the regret of the policy. We first control the second term $T_2$. We have, using $0\leq F_0, F_1 \leq 1$ \eqref{def:parameter-space}, 
\begin{align}
T_2 &\leq  \sum_{i \in [K]} N_i   \mathbb{P}[G_{i-1}^c] +  \sum_{i \in [K]} \mathbb{E}\Big[ \mathbb{1}( \{ \tau \leq i \} \cup G_{i-1}) \sum_{l = \bar{N}_{i-1}+1 }^{\bar{N}_i} [ F_*(X_{l}) - F_{A_l}(X_l)]  \Big] . \label{eq:T_2int}
\end{align} 
Observe that by our choice $\alpha_T = T^{-2}$ in our policy design, $\mathbb{P}[G_{i-1}^c]   \leq T^{-2}$ using Proposition \ref{prop:isoband}. Therefore, $\sum_{i \in [K]} N_i \mathbb{P}[G_{i-1}^c] \leq \sum_{ i\in [K]} N_i . T^{-2} = \frac{1}{T}$. To control the second term in \eqref{eq:T_2int}, note that on the event $G_{i-1}$, the $l^{th}$ observation contributes non-trivial regret only if $\{X_l \in \Unc^{(i-1)}\}$. Thus for $i \in [K]$, defining the sigma-field
$$\Sigma_i = \sigma(\{X_i,Y_{A_i}(X_i): 1 \leq i \leq \bar{N}_i\}),$$ 
we can control the second term as follows. 
\begin{align}
&\sum_{i \in [K]} \mathbb{E}\Big[ \mathbb{1}( \{ \tau \leq i \} \cup G_{i-1}) \sum_{l = \bar{N}_{i-1}+1 }^{\bar{N}_i} [ F_*(X_{l}) - F_{A_l}(X_l)]  \Big]  \nonumber \\
& \leq  \sum_{i \in [K]} \mathbb{E}\Big[ \mathbb{1}( \{ \tau \leq i \} \cup G_{i-1}) \sum_{l = \bar{N}_{i-1}+1 }^{\bar{N}_i} \mathbb{P}_i [X_l \in \Unc^{(i-1)}] \Big], \nonumber 
\end{align} 
where $\mathbb{P}_i[\cdot]$ and $\mathbb{E}_i[\cdot]$ denote the conditional probability and expectation respectively, conditioned on $\Sigma_{i-1}$. Note that conditioned on $\Sigma_{i-1}$, $\mathbb{P}_i[X_l \in \Unc^{(i-1)}] = \Leb(\Unc^{(i-1)})$ for all $l \in [\bar{N}_{i-1}, \bar{N}_i]$. Thus 
\begin{align}
&\sum_{i \in [K]} \mathbb{E}\Big[ \mathbb{1}( \{ \tau \leq i \} \cup G_{i-1}) \sum_{l = \bar{N}_{i-1}+1 }^{\bar{N}_i} [ F_*(X_{l}) - F_{A_l}(X_l)]  \Big]  \nonumber \\ 
&\leq \sum_{i \in [K]} N_i \mathbb{E}\Big[ \mathbb{1}( \tau \leq i ) \Leb(\Unc^{(i-1)})] \leq \sqrt{T} \log T, \nonumber 
\end{align} 
where the last inequality uses the definition of $\tau$ in \eqref{eq:stopping_time}. Plugging these bounds back into \eqref{eq:T_2int}, we have, 
\begin{align}
T_2 \leq \frac{1}{T} + \sqrt{T} \log T. \label{eq:T_2bound}
\end{align}

\noindent
We now turn to the main term in \eqref{eq:regret_decomp1}. 
For any $i \in [K]$ and $j \in \{0,1\}$, define $$M_i^{j} = \sum_{l = \bar{N}_{i-1}+1 }^{\bar{N}_i} \mathbb{1} (X_l \in \Unc^{(i - 1)}, A_l = j).$$ 
In words, $M_i^{j}$ counts the datapoints in epoch $i$ with contexts in $\Unc^{(i-1)}$ for which arm $j $ is pulled. 
Observe that given $\Sigma_{i - 2}$, we have, for $j \in \{0,1\}$,
\begin{align}
M^{j}_{i-1} \sim \mathrm{Bin}\Big(N_{i-1}, \mathrm{Leb}\Big( \Unc^{(i - 2)} \Big)/2 \Big). \nonumber 
\end{align}

\noindent
Thus for $i < \tau$, $\mathbb{E}_{i - 2} M^{j}_{i-1} \geq N_{i-1} \cdot \frac{ \log T}{2 \sqrt{T}} \geq \frac{1}{2} \log T$ if $N_{i-1} \geq \sqrt{T}$ which is true by our choice $N_1 \geq \sqrt{T}.$

For an appropriately small constant $0 < c < 1$ to be chosen later, define the events 
$$H_i = \cap_{j = 0}^{1} \{M^{j}_{i} > c \mathbb{E}_{i - 1} M^{j}_{i}\} ,\,\,\,\, i \in [K].$$

Using $0\leq F_0, F_1 \leq1$ and recalling the definition of $R_i$ from \eqref{eq:regret_decomp1}, we have, 
\begin{align}
\mathbb{E} \sum_{i\in[K]}  R_i  &\leq  \mathbb{E} \sum_{i \in [K]}  R_i  \mathbb{1}(H_{i - 1}) + \mathbb{E} \sum_{i \in [K]} N_i \mathbb{1}(H_{i - 1})^c \mathbb{1}(i < \tau) \nonumber \\
&\leq \mathbb{E} \sum_{i \in [K]}  R_i  \mathbb{1}(H_{i - 1}) + \mathbb{E} \sum_{i \in [K]} N_i  \mathbb{1}(i-1 < \tau) \mathbb{P}_{i-1}(H_{i - 1}^c ) \nonumber \\
&\leq \mathbb{E} \sum_{i \in [K]}  R_i  \mathbb{1}(H_{i - 1}) +  T^{-1}, \label{eq:main_term_bound} 
\end{align}
where the last inequality follows by conditioning on $\Sigma_{i - 2}$, noting that $\mathbb{1}(i-1 < \tau)$ is $\Sigma_{i-2}$ measurable, and then applying the binomial tail probability bound given in Lemma~\ref{lem:bin_dev}. The value $c$ is chosen so that $\mathbb{E}  \mathbb{1}(i-1 < \tau) \mathbb{1}(H_{i - 1})^{c} \leq T^{-2}.$

%


Next, we observe that
\begin{align}
&R_i \mathbb{1}(H_{i - 1}) = \sum_{l = \bar{N}_{i-1}+1 }^{\bar{N}_i} [ F_*(X_{l}) - F_{A_l}(X_{l})] \mathbb{1}(G_{i - 1} \cap \{X_l \in \Unc^{(i - 1)} \} \cap \{ i < \tau \} \cap H_{i - 1}) \nonumber\\
&\leq \sum_{l = \bar{N}_{i-1}+1 }^{\bar{N}_i} [(U^{(i-1)}_1(X_{l}) - L^{(i-1)}_1(X_{l})) + (U^{(i-1)}_0(X_{l}) - L^{(i-1)}_0(X_{l}) )] \mathbb{1}( \{X_l \in \Unc^{(i - 1)} \} \cap H_{i - 1}). \label{eq:int2} 
\end{align}

The inequality above holds by the following argument: on the event $G_{i - 1}$ and $\{X_{l} \in \Unc^{(i - 1)}\}$, assuming $F_1(X_{l}) \geq F_0(X_{l})$,

\begin{align*}
&F_{*}(X_l) - F_{A_l} (X_{l}) \leq F_1(X_{l}) - F_0(X_{l}) \nonumber \\
& \leq (U^{(i-1)}_1(X_{l}) - L^{(i-1)}_1(X_{l})) + (U^{(i-1)}_0(X_{l}) - L^{(i-1)}_0(X_{l}) ) + (L^{(i-1)}_1(X_{l}) - U^{(i-1)}_0(X_{l})) \nonumber \\
&\leq (U^{(i-1)}_1(X_{l}) - L^{(i-1)}_1(X_{l})) + (U^{(i-1)}_0(X_{l}) - L^{(i-1)}_0(X_{l}) ) , \nonumber 
\end{align*}
where the last inequality follows as $X_l \in \Unc^{(i - 1)}$. The same chain of inequalities hold if $F_1(X_{l}) < F_0(X_{l}).$ 
Plugging this back into \eqref{eq:int2}, we have, 
\begin{align*}
&R_i  \mathbb{1}(H_{i - 1}) \leq \\& \sum_{l = \bar{N}_{i-1}+1 }^{\bar{N}_i} [(U^{(i-1)}_1(X_{l}) - L^{(i-1)}_1(X_{l})) + (U^{(i-1)}_0(X_{l}) - L^{(i-1)}_0(X_{l})] \mathbb{1}(\{X_l \in \Unc^{(i - 2)} \} \cap H_{i - 1}) 
\end{align*}
since $\Unc^{(i - 1)} \subseteq \Unc^{(i - 2)}$.  Evaluating the conditional expectation given $\Sigma_{i-2}$ we get
\begin{align*}
&\mathbb{E}_{i - 2} R_i  \mathbb{1}(H_{i - 1}) \leq \\&\mathbb{E}_{i - 2} \sum_{l = \bar{N}_{i-1}+1 }^{\bar{N}_i} [(U^{(i-1)}_1(X_{l}) - L^{(i-1)}_1(X_{l})) + (U^{(i-1)}_0(X_{l}) - L^{(i-1)}_0(X_{l})] \mathbb{1}(\{X_l \in \Unc^{(i - 2)} \} \cap H_{i - 1}). 
\end{align*}


\noindent
Conditioning further on $M_{i - 1}^{1}$ we can apply our confidence width bound as given in Proposition~\ref{prop:isoband} with $A = \Unc^{(i - 2)}$ to obtain that 

\begin{align*}
&\mathbb{E}\Big[ \sum_{l = \bar{N}_{i-1}+1 }^{\bar{N}_i} [(U^{(i-1)}_1(X_{l}) - L^{(i-1)}_1(X_{l})] \mathbb{1}(X_l \in \Unc^{i - 2}) \mathbb{1}(H_{i - 1})  \mathbb{1}(i < \tau) \Big| \Sigma_{i-2}, M_{i - 1}^{1} \Big]  \\
& \leq  N_i \Leb(\Unc^{(i - 2)}) \big[C \Gamma_1 (\log T)^{3/2} \min\{(M^{1}_{i - 1})^{-1/3},\sqrt{k (M^{1}_{i - 1})^{-1/2} \log (eM^{1}_{i - 1}/k)}\} + \mathrm{Rem}\big] \mathbb{1}(M^{1}_{i-1} > c \mathbb{E}_{i - 2} M^{1}_{i-1})   \\
& \leq  N_i \Leb(\Unc^{(i - 2)}) \big[C \Gamma_1 (\log T)^{3/2} \min\{(N_{i - 1} \Leb(\Unc^{(i - 2)}))^{-1/3},\sqrt{k (N_{i - 1} \Leb(\Unc^{(i - 2)}))^{-1} \log T}\} + \mathrm{Rem}\big]
\end{align*}
where $\mathrm{Rem}$ consists of some lower order terms left to be verified by the reader. The other term can be handled analogously.

Now, taking expectation on both sides of the last display and bounding $\Leb(\Unc^{(i - 2)})$ by $1$, we can conclude that 
\begin{align*}
\mathbb{E} R_i  \mathbb{1}(H_{i - 1}) \leq C N_i \big[ \Gamma_1 (\log T)^{3/2} \min\{(N_{i - 1})^{-1/3},\sqrt{k (N_{i - 1})^{-1/2} \log T}\}
\end{align*}

\noindent 
Finally, we can sum over $i \in [K]$ and notice that $\sum_{i \in [K]} \frac{ N_i }{N_{i-1}^{1/3}} \leq C T^{2/3}$ and $\sum_{i \in [K]} \frac{ N_i }{N_{i-1}^{1/2}} \leq C T^{1/2}.$ The proof is complete by plugging this estimate back into \eqref{eq:main_term_bound} and \eqref{eq:regret_decomp1}.





\end{proof}

\subsubsection{Proof of Lemma \ref{lem:lower}} 
 \begin{proof}
 Our proof will follow the universal strategy of lower bounding the minimax regret in terms of an appropriate Bayes regret. We first derive the lower bound over $ \mathcal{D}(\tilde{C},L)$. Throughout, we assume that the error distribution is centered Gaussian, with an appropriate variance such that $\varepsilon$ satisfies Assumption A \eqref{eq:assum} with parameters $\tilde{C}$ and $L$. 
 
 We first construct an appropriate sub-class of our parameter space. We divide the interval $[0,1]$ into $m$-equal sub-intervals---the parameter $m>1$ will be chosen appropriately. For notational convenience, we set $\mathcal{I}_j = [(j-1)/m, j/m]$, $1\leq j \leq m$.  Set $\Sigma_m = \{\pm 1\}^m$. For each $\sigma \in \Sigma_m$, we construct a pair $(F_0, F_1)$ as follows: 
 if $\sigma_i =1$, we set $F_1(x) = i/m$ and $F_0(x) = (i-1)/m$ on the $i^{th}$ interval. On the contrary, if $\sigma_i = -1$, we set $F_1(x) = (i-1)/m$ and $F_0(x) = i/m$ on the $i^{th}$-interval. It is easy to see that the pair $(F_0^{\sigma}, F_1^{\sigma})$ constructed here is a valid parameter pair. Let $\mathcal{C}_m = \{(F_0^{\sigma}, F_1^{\sigma}): \sigma \in \Sigma_m\}$. Then we immediately have, 
 \begin{align}
 \max_{D \in \mathcal{D}(\tilde{C},L)} \mathrm{Regret}_T(\pi) \geq \frac{1}{2^m} \sum_{\sigma \in \Sigma_m} \mathrm{Regret}_T^{\sigma}(\pi), \label{eq:minimax_bayes} 
 \end{align}
 where $\mathrm{Regret}_T^{\sigma}(\pi)$ denotes the regret incurred by the policy $\pi$  under the parameter pair $(F_0^{\sigma}, F_1^{\sigma})$. Let $A^{\sigma}_*(i)$ denote the oracle optimal arm at round $i$ under the parameter pair $(F_0^{\sigma}, F_1^{\sigma})$. This implies
 \begin{align}
 \mathrm{Regret}_T^{\sigma}(\pi) &= \sum_{i=1}^{T} \mathbb{E}_{\sigma}[(F_*(X_i) - Y_i(A_i)  ) ] \nonumber\\
 &= \sum_{j=1}^{m} \sum_{i=1}^{T} \mathbb{E}_{\sigma} [|F_0^{\sigma}(X_i) - F_1^{\sigma}(X_i) | \mathbf{1}(X_i \in \mathcal{I}_j, A_i \neq A^{\sigma}_*(i)] \nonumber\\
 &\geq \frac{1}{m} \sum_{j=1}^{m} \sum_{i=1}^{T} \mathbb{P}_{\sigma}\Big[ X_i \in \mathcal{I}_j, A_i \neq \frac{1+\sigma_j}{2} \Big], \nonumber
 \end{align} 
 where the last inequality follows from our construction of $(F_0^{\sigma}, F_1^{\sigma})$. This automatically relates the regret of any policy to the corresponding \emph{inferior sampling rate}.  Plugging this back into \eqref{eq:minimax_bayes}, we obtain that 
 \begin{align}
 \max_{D \in \mathcal{D}(\tilde{C},L)} \mathrm{Regret}_T(\pi) &\geq \frac{1}{m \cdot 2^m} \sum_{t=1}^{T} \sum_{\sigma \in \Sigma_m} \sum_{j=1}^{m} \mathbb{P}_\sigma \Big[ X_t \in \mathcal{I}_j, A_t \neq \frac{1 + \sigma_j}{2} \Big] \nonumber \\
 &= \frac{1}{m \cdot 2^m} \sum_{j=1}^{m} \sum_{t =1}^{T} \sum_{\sigma_{[-j]} \in \Sigma_{m-1}} \sum_{i \in \{0,1\}} \mathbb{E}^{t-1}_{\sigma^{i}_{[-j]}} \mathbb{P}_X [ A_t \neq i , X_t \in \mathcal{I}_j], \nonumber
  \end{align} 
  where $\sigma^{i}_{[-j]} = (\sigma_1, \cdots, \sigma_{j-1}, 2i-1, \sigma_{j+1}, \cdots, \sigma_m)$, $\mathbb{E}^{t-1}_{\sigma}[\cdot]$ represents the joint distribution of the process over the first $(t-1)$ rounds, and $\mathbb{P}_X$ denotes the law of the new observation $X_t$. Further simplifying, we have, 
  \begin{align}
  \max_{D \in \mathcal{D}(\tilde{C},L)} \mathrm{Regret}_T(\pi) \geq \frac{1}{m^2 \cdot 2^m} \sum_{j=1}^m \sum_{t=1}^{T} \sum_{\sigma_{[-j]} \in \Sigma_{m-1}} \sum_{i \in \{0,1\}} \mathbb{E}^{t-1}_{\sigma^{i}_{[-j]}} \mathbb{P}_X^j [ A_t \neq i ], \nonumber 
  \end{align} 
  where $\mathbb{P}_X^j$ denotes the conditional distribution of $X_t$ given that $\{X_t \in \mathcal{I}_j\}$. We now observe that $\sum_{i \in \{0,1\}} \mathbb{E}^{t-1}_{\sigma^{i}_{[-j]}} \mathbb{P}_X^j [ A_t \neq i ]$ is the sum of Type I and Type II errors in a binary hypothesis testing problem. Thus using \cite[Theorem 2.2(iii)]{Tsybakovbook}, we have, 
  \begin{align}
  \max_{D \in \mathcal{D}(\tilde{C},L)} \mathrm{Regret}_T(\pi)  &\geq \frac{1}{m^2 \cdot 2^m} \sum_{j=1}^m \sum_{t=1}^{T} \sum_{\sigma_{[-j]} \in \Sigma_{m-1}} \exp\Big(- \mathrm{D}_{KL} (\mathbb{P}^{t-1}_{\sigma^1_{[-j]}} \times \mathbb{P}_X^j \|  \mathbb{P}^{t-1}_{\sigma^{0}_{[-j]}} \times \mathbb{P}_X^j ) \Big) \nonumber \\
  &= \frac{1}{m^2 \cdot 2^m} \sum_{j=1}^m \sum_{t=1}^{T} \sum_{\sigma_{[-j]} \in \Sigma_{m-1}} \exp\Big(- \mathrm{D}_{KL} (\mathbb{P}^{t-1}_{\sigma^1_{[-j]}}  \|  \mathbb{P}^{t-1}_{\sigma^{0}_{[-j]}} ) \Big) \label{eq:int_lower_bdd} 
  \end{align} 
  using the independence of $X_t$ and the past data. At this point, we require an upper bound on $\mathrm{D}_{KL} (\mathbb{P}^{t-1}_{\sigma^1_{[-j]}}  \|  \mathbb{P}^{t-1}_{\sigma^{0}_{[-j]}} )$. Let $\mathscr{F}_t^{+}$ denote the filtration corresponding to the policy $\pi$. Using the chain rule for KL-divergence, we obtain that 
  \begin{align}
  \mathrm{D}_{KL} (\mathbb{P}^{t}_{\sigma^1_{[-j]}}  \|  \mathbb{P}^{t}_{\sigma^{0}_{[-j]}} ) =  \mathrm{D}_{KL} (\mathbb{P}^{t-1}_{\sigma^1_{[-j]}}  \|  \mathbb{P}^{t-1}_{\sigma^{0}_{[-j]}} ) + \mathbb{E}^{t-1}_{\sigma^{1}_{[-j]} } \mathbb{E}_X \Big[ \mathrm{D}_{KL} \Big( \mathbb{P}_{\sigma^{1}_{[-j]} }^{Y_t(A_t) | \mathscr{F}_{t-1}^+} \| \mathbb{P}_{\sigma^{0}_{[-j]} }^{Y_t(A_t) | \mathscr{F}_{t-1}^+}  \Big)  \Big]. \nonumber 
  \end{align} 
  Observe that the second term has non-zero contribution to the divergence provided $X_t \in \mathcal{I}_j$. In this case, the divergence is bounded by that between two gaussian distributions with means $i/m$ and $(i-1)/m$ respectively, and with the same variance. Thus there exists a universal constant $C'>0$ (independent of $T$) such that 
  \begin{align}
  \mathrm{D}_{KL} (\mathbb{P}^{t}_{\sigma^1_{[-j]}}  \|  \mathbb{P}^{t}_{\sigma^{0}_{[-j]}} ) \leq   \mathrm{D}_{KL} (\mathbb{P}^{t-1}_{\sigma^1_{[-j]}}  \|  \mathbb{P}^{t-1}_{\sigma^{0}_{[-j]}} ) + \frac{C'}{m^3}. \nonumber 
  \end{align} 
  By induction, we obtain that $ \mathrm{D}_{KL} (\mathbb{P}^{t}_{\sigma^1_{[-j]}}  \|  \mathbb{P}^{t}_{\sigma^{0}_{[-j]}} ) \leq C' \frac{t}{m^3}$. Plugging this back into \eqref{eq:int_lower_bdd}, 
  \begin{align}
  \max_{D \in \mathcal{D}(\tilde{C},L)} \mathrm{Regret}_T(\pi) &\geq \frac{1}{m^2 \cdot 2^m} \sum_{j=1}^m \sum_{t=1}^{T} \sum_{\sigma_{[-j]} \in \Sigma_{m-1}} \exp(- C \frac{T}{m^3} ) \geq \frac{T}{2 m} \exp(-C' \frac{T}{m^3} ). \nonumber 
  \end{align} 
  Finally, we choose $m = T^{1/3}$. This provides the lower bound $\max_{D \in \mathcal{D}(\tilde{C},L)} \mathrm{Regret}_T(\pi) \geq \frac{1}{2} T^{2/3} \exp(-C')$. This completes the proof for $\mathcal{D}(\tilde{C},L)$.
  
  The lower bound for $\mathcal{D}^{(k)}(\tilde{C},L)$ is relatively straightforward. Assume again that the noise distribution is centered Gaussian with an appropriate variance. Further, assume that $F_0$, $F_1$ are piecewise constant on the intervals $[(i-1)/k, i/k]$ for $i \in \{1, \cdots, k\}$. As the intervals of constant value are known, this corresponds directly to a contextual bandit problem with $k$-discrete arms. One can directly adapt existing lower bound arguments (see \cite[Chapter 2]{slivkins2019introduction}) to see that each arm incurs a $\Theta(\sqrt{T/k})$ regret, and the total regret must be at least $c \sqrt{kT}$ for some constant $c>0$ (independent of $T$).   
   \end{proof}

\subsection{Confidence band results}

\subsubsection{Proof of Theorem~\ref{thm:valid}} 
The proof will proceed via two intermediate lemmas. 
\begin{lemma}\label{lem:determ}
Let $y = \theta^* + \varepsilon$ and $\hat{\theta}$ denote the isotonic $\tau$ quantile regression fit defined by~\eqref{eq:estimator}. Then the following pointwise inequality holds deterministically for all $i \in [n]$,
\begin{equation*}
\theta_i^* + \tau(\varepsilon_{i: \hat{u}_i}) \leq \hat{\theta}_i \leq \theta^*_i + \tau(\varepsilon_{\hat{l}_i:i}).
\end{equation*}
\end{lemma}

\begin{proof}
	For $\epsilon>0$ sufficiently small, consider an alternative estimator $\tilde{\theta}_j = \hat{\theta}_j + \epsilon$ for all $ j \in [i, \hat{u}_i]$, and $\tilde{\theta}_j = \hat{\theta}_j$ otherwise. As $\tilde{\theta} \in \mathcal{M}$, by the optimality of $\hat{\theta}$,  
	\begin{align}
	\sum_{k=1}^{n} \rho_{\tau}(y_k - \hat{\theta}_k) \leq \sum_{k=1}^{n} \rho_{\tau}(y_k - \tilde{\theta}_k) 
	\implies  0 \leq \sum_{k=i}^{\hat{u}_i} \rho_{\tau}(y_k - \hat{\theta}_k - \epsilon) -  \sum_{k=i}^{\hat{u}_i} \rho_{\tau}(y_k - \hat{\theta}_k).   \nonumber 
	\end{align} 
	
	\noindent 
	Next, we observe that 
	\begin{equation*}
	\rho_{\tau}(x - \epsilon) - 	\rho_{\tau}(x) = 
	\begin{cases}
	- \tau \epsilon, & \text{if} \:\: x - \epsilon > 0 \\
	(1 - \tau) \epsilon , & \text{if} \:\:x \leq 0 \\
	\end{cases}
	\end{equation*}
	
	\noindent 
	Using the two displays above, after dividing by $\epsilon > 0$ and setting $\epsilon\to 0$, we obtain
	\begin{align*}
	&(1 - \tau) \sum_{k=i}^{\hat{u}_i} \mathbb{1}(y_k - \hat{\theta}_k \leq 0) \geq \tau \sum_{k=i}^{\hat{u}_i} \mathbb{1}(y_k - \hat{\theta}_k > 0) 
	\implies  \tau \leq \frac{\sum_{k=i}^{\hat{u}_i} \mathbb{1}(y_k - \hat{\theta}_k \leq 0)}{\hat{u}_i - i + 1}
	\implies  \tau((y_k - \hat{\theta}_i)_{i : \hat{u}_i}) \leq 0
	\end{align*}
	where the last implication uses the fact that $\hat{\theta}_k = \hat{\theta}_i$ for $i \leq k \leq \hat{u}_i.$
	Thus we obtain, 
	\begin{align}
	\tau(y_{i:\hat{u}_i}) \leq \hat{\theta}_i \implies \theta_i^* + \tau(\varepsilon_{i: \hat{u}_i}) \leq \hat{\theta}_i.  \nonumber 
	\end{align}
	Similarly, one can obtain $\hat{\theta}_i \leq \theta^*_i + \tau(\varepsilon_{\hat{l}_i:i})$. This completes the proof. 
\end{proof}

\begin{lemma}\label{lem:hoeffding}
Suppose $\varepsilon_1,\dots,\varepsilon_n$ are i.i.d random variables with cdf $F$ satisfying $F(0) = \tau.$ In addition, suppose $F$  satisfies Assumption (A) as in~\eqref{eq:assum}. Also suppose $\Gamma_1$,$\Gamma_2$ are constants chosen to satisfy~\eqref{eq:cond1} and~\eqref{eq:cond2}. Then the following is true:
\begin{equation*}
\mathbb{P}\left(\max_{1 \leq k,l \leq n, l - k + 1 \geq \Gamma_2 \log n} \tau(\varepsilon_{k:l}) \sqrt{l - k + 1} \leq \Gamma_1 \sqrt{\log n}\right) \geq 1 - \alpha.
\end{equation*}
\end{lemma}

\begin{proof}
	Fix any pair of integers $1 \leq k,l \leq n$ such that $l - k + 1 \geq \Gamma_2 \log n.$ Now for any $t > 0$, by Hoeffding's inequality \cite{hoeffding1994probability} we have 
\begin{align}
\mathbb{P}[\tau(\varepsilon_{k:l}) > t] &= \mathbb{P}\Big[ \sum_{j=k}^{l} \mathbb{1}(\varepsilon_j > t) > ( l -k +1) (1-\tau) \Big]  \nonumber \\
&= \mathbb{P}\Big[ \sum_{j=k}^{l} (\mathbb{1}(\varepsilon_j > t)  - (1-F(t)) >  (l-k+1) (F(t) - F(0))    \Big] \nonumber \\
& \leq \exp\left(-2 (l - k + 1) (F(t) - F(0))^2\right) \nonumber 
\end{align}

Now set $t = \frac{\Gamma_1 \sqrt{\log n}}{\sqrt{l - k + 1}}$ and note that by~\eqref{eq:cond2} and the fact that $l - k + 1 \geq \Gamma_2 \log n$ this choice of $t$ lies between $0$ and $L.$  Therefore, we can further conclude that 
\begin{align*}
\mathbb{P}[\tau(\varepsilon_{k:l}) > t] \leq \exp\left(-2 \tilde{C}^2 \Gamma_1^2 \log n\right) = (\frac{1}{n})^{2 \tilde{C}^2 \Gamma_1^2} \leq \frac{\alpha}{n^2}.
\end{align*}
where the last inequality follows due to~\eqref{eq:cond1}. Since there are at most $n^2$ pairs of $(k,l)$ to consider, a union bound now finishes the proof of this lemma. 
\end{proof}

\begin{proof}[Proof of Theorem~\ref{thm:valid}]
Armed with the two lemmas above, we can now finish the proof of Theorem~\ref{thm:valid}. Recall the set $\hat{G}$ from the construction of the confidence band. Fix any $i \in [n].$ There are two cases to consider. 

\noindent 
\textbf{CASE $1$}: Suppose $i \in \hat{G}.$ By Lemma~\ref{lem:coverage} we have 
$$\hat{\theta}_i -  \tau(\varepsilon_{\hat{l}_i:i}) \leq \theta^*_i \leq \hat{\theta}_i  + \tau(\varepsilon_{i: \hat{u}_i}).$$ Now because $i \in \hat{G}$ we have $\min\{\hat{u}_i - i + 1, i - \hat{l}_i + 1\} \geq \Gamma_2 \log n.$ This allows us to use Lemma~\ref{lem:hoeffding} to conclude 
$$\hat{\theta}_i - \frac{\Gamma_1 \sqrt{\log n}}{\sqrt{i - \hat{l}_i + 1,}} \leq \hat{\theta}_i -  \tau(\varepsilon_{\hat{l}_i:i}) \leq \theta^*_i \leq \hat{\theta}_i  + \tau(\varepsilon_{i: \hat{u}_i}) \leq \hat{\theta}_i + \frac{\Gamma_1 \sqrt{\log n}}{\sqrt{\hat{u}_i - i + 1,}}.$$
Therefore, we have now established the desired coverage statement on $\hat{G}.$

\noindent 
\textbf{CASE $2$}: Suppose $i \in \hat{G}^{c}.$ If the set $\{j \in \hat{G}, j > i\}$ is empty then $\hat{\theta}^{(u)}_i = 1$ and trivially we have $\theta^*_i \leq \hat{\theta}^{(u)}_i.$ If the set  $\{j \in \hat{G}, j > i\}$ is non empty then $\hat{k}_{u}(i) \in \hat{G}$ is well defined and we have $\theta^*_i \leq \theta^*_{\hat{k}_{u}(i)} \leq \hat{\theta}^{(u)}_{\hat{k}_{u}(i)} = \hat{\theta}^{(u)}_{i}.$ A similar argument would also show that $\hat{\theta}^{(l)}_{i} \leq \theta^*_i.$ This establishes the desired coverage statement on $\hat{G}^{c},$ and completes the proof.


\end{proof}

\subsubsection{Proof of Theorem \ref{thm:width}}

The proof of Theorem \ref{thm:width} heavily rests on the following proposition which bounds the number of pieces of the isotonic quantile estimator $\hat{\theta}$ defined in~\eqref{eq:estimator}. 
\begin{proposition}\label{prop:pieces}
	Let $k$ denote the number of constant pieces of $\theta^*$ and $\hat{k}$ denote the number of constant pieces of $\hat{\theta}.$ Under the same assumptions as in Theorem~\ref{thm:width} we have the bound
	\begin{equation*}
	\E \hat{k} \leq C \min\{n^{1/3}, k \log n\}.
	\end{equation*}
\end{proposition}

Using the above proposition, we now present the proof of Theorem~\ref{thm:width}.

\begin{proof}[Proof of Theorem~\ref{thm:width}]
	We will only show how to bound $\mathbb{E} \sum_{i = 1}^{n} (\hat{\theta}^{(u)}_j - \hat{\theta}_j)$ as the other term $\mathbb{E}  \sum_{i = 1}^{n} (\hat{\theta}_j - \hat{\theta}^{(l)}_j)$ can be controlled similarly.

	Recall the set $\hat{G}$ from the construction of the confidence band. 
	We decompose  $$\sum_{i = 1}^{n} (\hat{\theta}^{(u)}_j - \hat{\theta}_j) = \underbrace{\sum_{i \in \hat{G}} (\hat{\theta}^{(u)}_j - \hat{\theta}_j)}_{T_1} + \underbrace{\sum_{i \in \hat{G}^{c}} (\hat{\theta}^{(u)}_j - \hat{\theta}_j)}_{T_2}.$$
	We will first bound $T_1.$ Note that when $j \in \hat{G}$ the difference $\hat{\theta}^{(u)}_j - \hat{\theta}_j$ is of the form $\Gamma_1 \sqrt{\log n} \cdot \frac{1}{\sqrt{l}}$ where $l$ is an integer between $1$ and $n_i$ and $n_i$ is the length of the constant block of $\hat{\theta}$ which contains $j.$ Therefore, denoting $H_1(m) = 1 + 1/\sqrt{2} + \dots +  1/\sqrt{m}$ for any integer $m$, we have

	\begin{align*}
	&T_1 \leq \Gamma_1 \sqrt{\log n} \sum_{i = 1}^{\hat{k}} H_1(n_i) \leq \Gamma_1 \sqrt{\log n} \sum_{i = 1}^{\hat{k}} \sqrt{n_i} \leq  \Gamma_1 \sqrt{\log n} \sqrt{n \hat{k}}
	\end{align*}
	where the last step follows using Jensen's inequality. 
	Taking expectation we can write
	\begin{align*}
	&\mathbb{E} T_1 \leq \Gamma_1 \sqrt{n\log n} \:\mathbb{E} \sqrt{\hat{k}} \leq \Gamma_1 \sqrt{n\log n} \:\sqrt{\mathbb{E} \hat{k}} \leq \Gamma_1 \sqrt{n\log n} \sqrt{\min\{n^{1/3}, k \log n\}} \leq \\& \Gamma_1 \min\{n^{2/3} \sqrt{\log n}, n^{1/2} \sqrt{k} \log n\}.
	\end{align*}
	where the second inequality follows from Jensen's inequality and the third inequality follows from Proposition~\ref{prop:pieces}.

	To bound $T_2$, observe that in each constant piece of $\hat{\theta}$ there are at most $2 \Gamma_2  \log n$ elements in $\hat{G}^{c}$ and therefore $|\hat{G}^c| \leq 2 \Gamma_2 \hat{k}\: \log n$ deterministically. Moreover, since $\hat{\theta}^{(u)} - \theta^*$ take values between $0$ and $1$ we can write
	\begin{align*}
	\E T_2 \leq \E \sum_{i \in \hat{G}^{c}} (\hat{\theta}^{(u)}_j - \hat{\theta}_j) \leq \E |\hat{G}^c| \leq \Gamma_2 \log n \:\:\E \hat{k} \leq \Gamma_2 \log n \min\{n^{1/3}, k \log n\}.
	\end{align*} 
	This finishes the proof of the theorem. 
\end{proof}

\noindent 
It remains to prove Proposition~\ref{prop:pieces}. This proof needs three ingredient lemmas. First observe that 
\begin{align}
\hat{k} = 1 + \sum_{i=2}^{n} \mathbf{1}(\hat{\theta}_{i-1} < \hat{\theta}_i) 
\implies  \mathbb{E}[\hat{k}] = 1 + \sum_{i=2}^{n} \mathbb{P}[\hat{\theta}_{i-1} < \hat{\theta}_i]. \nonumber 
\end{align} 
Thus in order to bound $\E \hat{k}$ it suffices to control $\mathbb{P}[\hat{\theta}_{i-1} < \hat{\theta}_i]$ for $2\leq i \leq n$. The following two lemmas ultimately identify an event which contains the event $\{ \hat{\theta}_{i-1} < \hat{\theta}_i\}$ and which can be explicitly written in terms of the error random variables $\varepsilon.$

\begin{lemma}
	\label{lem:pieces_control} 
	For $2\leq i \leq n$, 
	\begin{align}
	\{ \hat{\theta}_{i-1} < \hat{\theta}_i\} \subseteq \Big\{ \max_{l \leq i-1} \tau(y_{l:i-1}) < \min_{u\geq i} \tau(y_{i:u}) \Big\}. \nonumber 
	\end{align}
\end{lemma} 

\begin{proof}[Proof of Lemma \ref{lem:pieces_control}] 
	The proof follows once we establish the following inequalities. For any $l \leq i -1$ and $u \geq i$, 
	\begin{align}
	\tau(y_{l:i-1}) \leq \hat{\theta}_{i-1}, \,\,\,\,\, \hat{\theta}_i \leq \tau(y_{i:u}). \nonumber 
	\end{align}
	We establish these inequalities through a perturbative argument. Fix $l \leq i-1$, and construct a new vector $\tilde{\theta} \in \mathcal{M}$ as follows: $\tilde{\theta}_j = \hat{\theta}_j + \epsilon$ for $l \leq j \leq i-1$, while $\tilde{\theta}_j = \hat{\theta}_j$ otherwise. $\tilde{\theta}$ is a small perturbation of $\hat{\theta}$ as $\epsilon$ is taken to be small enough so that $\tilde{\theta} \in \mathcal{M}$. We emphasize that this perturbation is possible because $\hat{\theta}_{i-1} < \hat{\theta}_i$. Using the optimality of $\hat{\theta}$, we obtain that 
	\begin{align}
		&\sum_{j=1}^{n} \rho_{\tau}(y_j - \hat{\theta}_j) \leq \sum_{j=1}^{n} \rho_{\tau}(y_j - \tilde{\theta}_j) \nonumber\\
		&\sum_{j=l}^{i-1} \rho_{\tau}(y_j - \hat{\theta}_j) \leq \sum_{j=l}^{i-1} \rho_{\tau}(y_j - \hat{\theta}_j - \epsilon). \nonumber \\
	    \implies  &\tau((y_j - \hat{\theta}_j)_{l:i-1}) \leq 0 \nonumber 
	\end{align} 
	where the last line follows by arguments similar to the one made in the proof of Lemma~\ref{lem:determ}.
	The monotonicity of $\hat{\theta}$ implies that $\tau((y_j - \hat{\theta}_{i-1})_{l:i-1}) \leq 0$, which, in turn, implies $\tau(y_{l:i-1}) \leq \hat{\theta}_{i-1}$. A similar argument can be made to show that $\tau(y_{i:u}) \geq \hat{\theta}_{i}$. 
\end{proof}

\begin{lemma}
	\label{lemma:prob_pieces} 
	For any $1\leq l_i \leq i-1$ and $i \leq u_i \leq n$, 
	\begin{align}
	\Big\{\max_{l \leq i-1} \tau(y_{l:i-1}) < \min_{u\geq i} \tau(y_{i:u}) \Big\} \subseteq \Big\{\max_{l_i \leq l\leq i -1} \tau(\varepsilon_{l:i-1}) < \min_{i \leq u \leq u_i} \tau(\varepsilon_{i:u}) + \theta_{u_i}^* - \theta_{l_i}^*  \Big\}. \nonumber 
	\end{align} 
\end{lemma}

\begin{proof}[Proof of Lemma \ref{lemma:prob_pieces}]
	We note that for any $1\leq l_i \leq i-1$ and $i \leq u_i \leq n$, 
	\begin{align}
	\Big\{\max_{l \leq i-1} \tau(y_{l:i-1}) < \min_{u\geq i} \tau(y_{i:u}) \Big\} \subseteq \Big\{\max_{l_i \leq l \leq i-1} \tau(y_{l:i-1}) < \min_{i \leq u \leq u_i } \tau(y_{i:u}) \Big\}. \nonumber 
	\end{align} 
	Using the monotonicity of $\theta^*$, for $l_i \leq j$, $y_j = \theta^*_j + \varepsilon_j \geq \theta^*_{l_i} +\varepsilon_j$, while for any $j \leq u_i$, $y_j = \theta_j^* + \varepsilon_j \leq \theta_{u_i}^* + \varepsilon_j$. This implies 
	\begin{align}
	\Big\{\max_{l_i \leq l \leq i-1} \tau(y_{l:i-1}) < \min_{i \leq u \leq u_i } \tau(y_{i:u}) \Big\} \subseteq \Big\{\max_{l_i \leq l\leq i -1} \tau(\varepsilon_{l:i-1}) < \min_{i \leq u \leq u_i} \tau(\varepsilon_{i:u}) + \theta_{u_i}^* - \theta_{l_i}^*  \Big\}. \nonumber 
	\end{align} 
\end{proof} 

\noindent 
By Lemma~\ref{lem:pieces_control} and Lemma~\ref{lemma:prob_pieces} we have shown the event inclusion relation
\begin{equation*}
\{ \hat{\theta}_{i-1} < \hat{\theta}_i\} \subset \Big\{\max_{l_i \leq l\leq i -1} \tau(\varepsilon_{l:i-1}) + \max_{i \leq u \leq u_i} (1 -\tau)(-\varepsilon_{i:u})< \theta_{u_i}^* - \theta_{l_i}^*  \Big\}.
\end{equation*}

\noindent 
We wish to bound the probability of the event on the R.H.S of the above display. For this purpose we state our next lemma.


\begin{lemma}\label{lem:woodroofe_meyer_argument} 
	Fix $m,n \geq 1$. Let $X_1, \cdots, X_m, Y_1, \cdots, Y_n \sim F$ be iid. Then there exists $C>0$ (independent of $m$ and $n$) such that for all $z>0$, 
	\begin{align}
	\mathbb{P}\Big[ \max_{1\leq i \leq m} \tau(X_{1:i}) + \max_{1\leq j \leq n} (1- \tau)(-Y_{1:j}) \leq z\Big] \leq C \Big( \frac{1}{m} + \frac{1}{n} + z^2\Big). \nonumber 
	\end{align}
\end{lemma}

Lemma~\ref{lem:woodroofe_meyer_argument} is the key probabilistic result underlying our bound in Proposition~\ref{prop:pieces} of the average number of constant pieces in our isotonic quantile estimator. Lemma~\ref{lem:woodroofe_meyer_argument} can be thought of as the quantile analogue of Proposition $5$ in~\cite{MW00}. Our proof technique is an adaptation of the general proof strategy of Proposition $5$ in~\cite{MW00} to our setting where we require bounds on maxima of suitably defined random walks with non symmetric $\pm 1$ valued increments instead of Gaussian increments which was the case in~\cite{MW00}. The required bounds on maxima of these random walks are carried out in Lemmas~\ref{lem:RW1},~\ref{lem:RW2} in Section~\ref{sec:woodroofe}. For now, we give the proof of Proposition~\ref{prop:pieces} assuming Lemma~\ref{lem:woodroofe_meyer_argument}.

\begin{proof}[Proof of Proposition~\ref{prop:pieces}]
Combining the results of Lemmas~\ref{lem:pieces_control},~\ref{lemma:prob_pieces} and~\ref{lem:woodroofe_meyer_argument} we obtain the following bound:
\begin{equation*}
\mathbb{P}[\hat{\theta}_{i-1} < \hat{\theta}_i] \lesssim  \min_{1 \leq l_i \leq i - 1, i \leq u_i \leq n} \left(\frac{1}{u_i - i + 1} + \frac{1}{i - l_i + 1} + (\theta^*_{u_i} - \theta^*_{l_i})^2\right)
\end{equation*}
Summing over the indices of $i$ we get
\begin{align*}
\E \hat{k} \lesssim 1 + \sum_{i = 2}^{n} \min_{1 \leq l_i \leq i - 1, i \leq u_i \leq n} \left(\frac{1}{u_i - i + 1} + \frac{1}{i - l_i + 1} + (\theta^*_{u_i} - \theta^*_{l_i})^2\right) \lesssim &\\ 1 + \underbrace{\sum_{i = 2}^{n} \min_{i \leq u_i \leq n} \left(\frac{1}{u_i - i + 1} + (\theta^*_{u_i} - \theta^*_{i})^2\right)}_{A} + \underbrace{\sum_{i = 2}^{n} \min_{1 \leq l_i \leq i - 1} \left(\frac{1}{i - l_i + 1} + (\theta^*_{i} - \theta^*_{l_i})^2\right)}_{B}.
\end{align*}
We will now bound the term $A$. The term $B$ can be bounded similarly.

We derive two different bounds for A.
\begin{enumerate}
	\item Bound $1$:

For each $2 \leq i \leq n$ we set the value of $u_i$ so that 
\begin{equation*}
u_i^* = \max\Big\{j \geq i: (\theta^*_{j} - \theta^*_{i})^2 \leq \frac{1}{j - i + 1}\Big\}.
\end{equation*}
Then we have 
\begin{equation*}
A \leq \sum_{i = 2}^{n} \left(\frac{1}{u_i^* - i + 1} + (\theta^*_{u_i} - \theta^*_{i})^2\right).
\end{equation*}
Now, for $j = 1,2,\dots$, define the sets 
\begin{equation*}
\mathcal{C}_j = \{2 \leq i \leq n: 2^{j - 1} \leq u_i^* - i + 1 < 2^j\}
\end{equation*}
Then by definition of $u_i^*$ we can write
\begin{equation*}
A \leq 2 \sum_{i = 2}^{n} \frac{1}{u_i^* -i + 1} \leq 2 \frac{|\mathcal{C}_j|}{2^{j - 1}} \lesssim \sum_{j} \frac{|\mathcal{C}_j|}{2^{j}}
\end{equation*}
where $|\mathcal{C}_j|$ denotes the cardinality of $\mathcal{C}_j.$ 

Now we claim that we can bound the cardinality of $\mathcal{C}_j$ as follows:
\begin{equation}\label{eq:card}
|\mathcal{C}_j| \leq 2^j \min\{\sqrt{2^j}, \frac{n}{2^{j - 1}}\}.
\end{equation}

Modulo the above claim we obtain 
$A \lesssim \sum_{j} \min\{\sqrt{2^j}, \frac{n}{2^j}\} \lesssim n^{1/3}$.
The final bound in the display above is obtained upon observing that $(\sqrt{2)}^j < n/2^{j}$ as long as $(\sqrt{2})^j < n^{1/3}$, and the final conclusion follows upon summing the two geometric series separately. 

It now remains to prove the claim~\eqref{eq:card}. If $\mathcal{C}_j$ is empty there is nothing to prove. Otherwise, define the index $i_1 = \min\{2 \leq i \leq n: i \in \mathcal{C}_j\}.$ Define the interval $\mathcal{B}_1 = [i_1,u_{i_1}^* + 1].$ If $\mathcal{C}_j - \mathcal{B}_1$ is empty, then stop otherwise define the next index $i_2 = \min\{u_{i_1}^* + 1 < i \leq n: i \in \mathcal{C}_j\}$ and define the interval $\mathcal{B}_2 = [i_2,u_{i_2}^* + 1]$. Iterate this process till it stops to obtain intervals $\mathcal{B}_1,\mathcal{B}_2,\dots,\mathcal{B}_{l}$ say, where $l$ is the number of blocks obtained in this process. Note that these intervals satisfy the following properties: a) they are pairwise disjoint, b) their union covers $\mathcal{C}_j.$ c) each of the intervals $\mathcal{B}_j$ satisfy that $\theta_{\mathrm{last}(\mathcal{B}_j)}^* - \theta_{\mathrm{first}(\mathcal{B}_j)}^* > \sqrt{2^{-j}}$ where $\mathrm{first}(B_j),\mathrm{last}(B_j)$ are the first and last indices of the interval $\mathcal{B}_j$, d) Each of the intervals satisfy $|\mathcal{B}_j| \geq 2^{j - 1}.$ e) Each of the intervals satisfy $|\mathcal{B}_j| \leq 2^{j}.$

Property (d) implies that $l \leq \frac{n}{2^{j - 1}}$. Also, by property (c) and the fact that $\theta^*_n - \theta^*_1 \leq 1$ we have $l \leq \sqrt{2^j}.$ Therefore, we can conclude
$$|\mathcal{C}_j| \leq \sum_{i = 1}^{l} |\mathcal{B}_l| \leq 2^j \min\{\sqrt{2^j}, \frac{n}{2^{j - 1}}\}.$$

\item Bound $2$: 

For each $2 \leq i \leq n$ we set the value of $u_i^*$ so that 

\begin{equation*}
u_i^* = \max\{j \geq i: (\theta^*_{u_i} - \theta^*_{i}) = 0\}.
\end{equation*}

Then we have 
\begin{align*}
A \leq \sum_{i = 2}^{n} \left(\frac{1}{u_i^* - i + 1}\right) \leq \sum_{i = 1}^{k} H(n_i) \lesssim k \log n
\end{align*}
where $k$ equals the number of constant pieces of $\theta^*$ and $H(m) = 1 + 1/2 + \dots + 1/m$ is the $m$th term of the partial sums of the harmonic sequence. 

\end{enumerate}
\end{proof}

\subsubsection{Proof of Proposition~\ref{prop:isoband}}

\begin{proof}[Proof of Proposition~\ref{prop:isoband}]
Let $\xi_j = x_{(j)}$ denote the $j$th order statistic of the set $\{x_i\}_{i = 1}^{n}.$ 
To prove~\eqref{eq:confi}, we first observe that because of our piecewise constant interpolation scheme, it is enough to show  coverage only at the design points, 
$$\mathbb{P}\big(L(\xi_j) \leq f(\xi_j) \leq U(\xi_j) \:\: \forall j \in [n]\big) \geq 1 -  \alpha.$$
This is a direct consequence of Theorem~\ref{thm:valid}.

Now we will prove the second part of Proposition~\ref{prop:isoband}, namely~\eqref{eq:width}. Let us denote $$W = \big[(U(Z) - L(Z)) \mathbb{I}(Z \in A)\big].$$
We can immediately write $\mathbb{E} W= (\mathbb{E} W| \mathbb{I}\{Z \in A\}) \: \mathbb{P}(Z \in A)= (\mathbb{E} W| \mathbb{I}\{Z \in A\})\:\Leb(A)$ where $\Leb(A)$ stands for the Lebesgue measure of $A.$ Therefore, it suffices to control the conditional expectation $\mathbb{E} W| \mathbb{I}\{Z \in A\}.$
Observe that computing  the conditional expectation $\mathbb{E} W| \mathbb{I}\{Z \in A\}$ is same as computing the unconditional expectation $\mathbb{E}(U(Z) - L(Z)) $ when $Z$ is now drawn from $\Unif(A)$ instead of $\Unif(0,1).$ 
We write $(U(Z) - L(Z)) = U(Z) - f(Z) + f(Z) - L(Z).$ We will now bound $\mathbb{E}(U(Z) - f(Z))$ where $Z \sim \Unif(A)$. A similar bound holds for $\mathbb{E}(f(Z) - L(Z)).$
By definition, $U(Z) = U(Z^{+}).$ Therefore, we first write
\begin{equation}\label{eq:1}
 U(Z) - f(Z) = U(Z^{+}) - f(Z^{+}) +  f(Z^{+}) - f(Z).
\end{equation}
Now, let us denote $I_j = [\xi_{j - 1},\xi_j) \cap A$ for $j = 1,\dots,n + 1$ and $\xi_{0} = 0, \xi_{n + 1} = 1.$ Define the good set 
$$G = \{\max_{1 \leq j \leq n + 1} \Leb(I_j)\leq C \frac{\log n}{n} \}. $$

\noindent 
We can now write 
\begin{align*}
&\mathbb{E} \:[U(Z^{+}) - f(Z^{+})] = \mathbb{E}\:\sum_{j = 1}^{n + 1} [U(\xi_j) - f(\xi_j)]\:\mathbb{1}(Z \in I_j) \leq \\&\mathbb{E}\: \Big[\Big(\max_{1 \leq j \leq n + 1} \Leb( I_j) \Big) \big(\sum_{j = 1}^{n + 1} [U(\xi_j) - f(\xi_j)]\big) \mathbb{1}(G) \Big] +  \mathbb{E}\:\Big[\max_{1 \leq j \leq n + 1} [U(\xi_j) - f(\xi_j)] \mathbb{1}(G^c) \Big]  \leq\\&C \frac{\log n}{n} \:\mathbb{E} \sum_{j = 1}^{n} [U(\xi_j) - f(\xi_j)] + \mathbb{P} (G^{c}) \leq  C \log n \:\:b_n +  n^{-2}. 
\end{align*}
To obtain the last inequality above, we use 
Theorem~\ref{thm:width} and Lemma~\ref{lem:orderstat} (specifically Remark \ref{rem:gaps_union}) to control the first term. The second term is controlled using the trivial bound that $U(\xi_j) - f(\xi_j) \leq 1$ for all $1\leq j \leq n+1$. 
Also, we have
\begin{align*}
&\mathbb{E} \:[f(Z^{+}) - f(Z)] \leq \mathbb{E} \:[f(Z^{+}) - f(Z^{-})] = \mathbb{E} \sum_{j = 1}^{n + 1} [f(\xi_j) - f(\xi_{j - 1})]\:\: \Leb( I_j) \leq \mathbb{E} \max_{1 \leq j \leq n + 1} \Leb( I_j) \leq \\& C \frac{\log n}{n} + \mathbb{P}(G^c) \leq  C \frac{\log n}{n} + n^{-2}.
\end{align*}
Combining the last two displays with~\eqref{eq:1} finishes the proof. 

\end{proof}

\subsection{Proof of Lemma~\ref{lem:woodroofe_meyer_argument}}\label{sec:woodroofe}
In this section we give the proof of Lemma~\ref{lem:woodroofe_meyer_argument}. This will require two intermediate lemmas. 
\begin{lemma}\label{lem:RW1}
	\label{lem:left_tail}
	Let $X_1, \cdots, X_m \sim F$ be iid. 
	\begin{itemize}
		\item[(i)] There exists $C>0$ such that 
		\begin{align}
		\mathbb{P}\Big[ \max_{1\leq i \leq m} \tau(X_{1:i} ) \leq 0 \Big] \leq \frac{C}{\sqrt{m}}. \nonumber 
		\end{align} 
		
		\item[(ii)] Assume there exists a constant $\varepsilon_0>0$ such that for all $\varepsilon \in (0,\varepsilon_0]$,  $F(0) - F( -\varepsilon) > c \varepsilon$ for some $c>0$. 
		Then there exists $\varepsilon>0$ and $C>0$ such that 
		\begin{align}
		\mathbb{P}\Big[\max_{1\leq i \leq m} \tau(X_{1:i}) \leq - \varepsilon \Big] \lesssim \frac{1}{\sqrt{m}}  \exp(- C m \varepsilon^2). \nonumber 
		\end{align} 
	\end{itemize} 
\end{lemma} 

\begin{proof}[Proof of Lemma \ref{lem:left_tail}]
	\begin{itemize} 
		
		\item[(i)] We define $W_j = \mathbf{1}(X_j \leq 0)$. For any $1\leq i \leq m$, observe that 
		\begin{align}
		\Big\{ \tau(X_{1:i}) \leq 0 \Big\} = \Big\{ \sum_{j=1}^{i} (W_j - \tau) \geq 0 \Big\}. \nonumber 
		\end{align} 
		Define $S_i = \sum_{j=1}^{i} (W_j - \tau)$, and observe that  
		\begin{align}
		\mathbb{P}\Big[\max_{1\leq i \leq m} \tau(X_{1:i}) \leq 0 \Big] &= \mathbb{P}\Big[ \min_{1\leq i \leq m} S_i \geq 0 \Big] \sim \frac{C}{\sqrt{m}}, \nonumber 
		\end{align} 
		where the last display follows immediately from \cite[Chapter XII.7, Theorem 1a]{Feller2}.

		\item[(ii)]  Define $W_j = \mathbf{1}(X_j \leq - \varepsilon)$ and $V_j = \mathbf{1}(X_j \leq 0)$. For any $1\leq i \leq m$, observe that 
		\begin{align}
		\Big\{ \tau(X_{1:i}) \leq - \varepsilon \Big\} = \Big\{ \sum_{j=1}^{i} (W_j - \tau) \geq 0 \Big\}. \nonumber 
		\end{align} 
		Define $S_i = \sum_{j=1}^{i} (W_j - \tau)$, and observe that 
		\begin{align} 
		\Big\{ \max_{1\leq i \leq m} \tau(X_{1:i}) \leq - \varepsilon \Big\} = \Big\{ \min_{1\leq i \leq m} S_i \geq 0 \Big\}. \nonumber 
		\end{align} 
		Let $\{\xi_1, \cdots, \xi_m\}$ be iid random variables such that under $\mathbb{P}_{-\varepsilon}$, $\mathbb{P}_{-\varepsilon}[\xi_1 =1]= F(-\varepsilon)$, while under $\mathbb{P}_0[\xi_1=1] = F(0) = \tau$. Using the discussion above, we have, 
		\begin{align}
		\mathbb{P}\Big[\max_{1\leq i \leq m} \tau(X_{1:i}) \leq - \varepsilon \Big] &= \mathbb{P}\Big[ \min_{1\leq i \leq m} S_i \geq 0 \Big] \nonumber \\ 
		&= \mathbb{P}_{-\varepsilon}\Big[ \min_{1\leq i \leq m} \sum_{j=1}^{i} (\xi_j - \tau) \geq 0] \nonumber \\
		&= \mathbb{E}_0 \Big[ \mathbf{1}\Big(\min_{1\leq i \leq m} \sum_{j=1}^{i} (\xi_j - \tau) \geq 0 \Big)  \frac{\mathrm{d}\mathbb{P}_{-\varepsilon}}{\mathrm{d}\mathbb{P}_0}\Big]. \nonumber 
		\end{align} 
		Observe that 
		\begin{align}
		\frac{\mathrm{d}\mathbb{P}_{-\varepsilon}}{\mathrm{d}\mathbb{P}_0}= \Big( \frac{F(-\varepsilon)/(1-F(-\varepsilon))}{\tau/(1-\tau)} \Big)^{\sum_{j=1}^{m} (\xi_j - \tau)} \cdot \exp\Big[ m \Big( \tau \log \frac{F(-\varepsilon)}{\tau} + (1- \tau) \log \frac{1- F(-\varepsilon)}{1- \tau} \Big) \Big]. 
		\end{align}
		Using the fact that $F(x)/(1-F(x))$ is non-decreasing, we have, 
		\begin{align} 
		\mathbb{P}\Big[\max_{1\leq i \leq m} \tau(X_{1:i}) \leq - \varepsilon \Big] &\leq \mathbb{E}_0 \Big[ \mathbf{1}\Big(\min_{1\leq i \leq m} \sum_{j=1}^{i} (\xi_j - \tau) \geq 0 \Big)  \Big] \exp\Big[ m \Big( \tau \log \frac{F(-\varepsilon)}{\tau} + (1- \tau) \log \frac{1- F(-\varepsilon)}{1- \tau} \Big) \Big] \nonumber \\
		&\lesssim \frac{1}{\sqrt{m}} \exp(-m (h(\tau) - h(F(-\varepsilon))), \nonumber 
		\end{align} 
		where we set $h(a) = \tau \log a + (1-\tau) \log (1-a)$. The function $h$ is concave, and is maximized at $a= \tau$. This implies, 
		\begin{align}
		h(\tau) - h(F(-\varepsilon)) = - \frac{h''(\xi)}{2} (\tau - F(-\varepsilon))^2 \geq C \varepsilon^2. \nonumber
		\end{align} 
		for $\varepsilon>0$ small enough. This completes the proof. 
		

	\end{itemize} 
\end{proof}

\begin{lemma}\label{lem:RW2}
	\label{lemma:density_near_origin}
	Assume that there exists $z_0>0$ such that for all $z \in (0, z_0]$, $F(z) - F(0) \leq cz$ for some $c>0$. Then there exists $C>0$ such that for all $0\leq z \leq 1$, 
	\begin{align}
	\mathbb{P}\Big[0 < \max_{1\leq i \leq m} \tau(X_{1:i}) \leq z \Big] \leq Cz. \nonumber 
	\end{align} 
\end{lemma}

\begin{proof}[Proof of Lemma \ref{lemma:density_near_origin}]
	Note that by choosing $C>0$ sufficiently large if necessary, the bound follows trivially in any interval $[z_0,1]$. Thus for the subsequent proof, we assume, without loss of generality, that $z \in [0,z_0]$ for some $z_0$ sufficiently small, to be chosen suitably. Fix $z>0$. First observe that 
	\begin{align}
	\mathbb{P}\Big[0 < \max_{1\leq i \leq m} \tau(X_{1:i}) \leq z \Big] &\leq \mathbb{P}\Big[-w < \max_{1\leq i \leq m} \tau(X_{1:i}) \leq z \Big] \nonumber \\
	&= \mathbb{P}\Big[\max_{1\leq i \leq m} \tau(X_{1:i}) \leq z \Big] - \mathbb{P}\Big[\max_{1\leq i \leq m} \tau(X_{1:i}) \leq -w\Big] \nonumber 
	\end{align} 
	for any $w>0$. Define $W_j = \mathbf{1}(X_j \leq z)$ and $V_j = \mathbf{1}(X_j \leq - w)$. We define $S^{(1)}_i = \sum_{j=1}^{i} (W_j - \tau)$ and $S^{(2)}_i = \sum_{j=1}^{i} (V_j - \tau)$. Using the same observation as in the proof of Lemma \ref{lem:left_tail}, we have,  
	\begin{align}
	\mathbb{P}\Big[0 < \max_{1\leq i \leq m} \tau(X_{1:i}) \leq z \Big] \leq \mathbb{P}\Big[\min_{1\leq i \leq m} S^{(1)}_i \geq 0 \Big] - \mathbb{P}\Big[\min_{1\leq i \leq m} S^{(2)}_i \geq 0 \Big]. \nonumber 
	\end{align} 
	Observe that $\mathbb{E}[W_j - \tau] = F(z) - \tau >0$, while $\mathbb{E}[V_j - \tau] = F(-w) - \tau <0$---thus $S^{(1)}$ and $S^{(2)}$ are biased random walks in the positive and negative direction respectively. Note that the two terms above correspond to the probability of the same event---albeit under different probability measures. It will be convenient for us to denote the distribution of the $W_j$ variables as $\mathbb{P}_z$, and that of the $V_j$ variables as $\mathbb{P}_{-w}$. Under this new notation, 
	\begin{align}
	\mathbb{P}\Big[0 < \max_{1\leq i \leq m} \tau(X_{1:i}) \leq z \Big] \leq \mathbb{P}_{z} \Big[\min_{1\leq i \leq m} \sum_{j=1}^{i} (\xi_j - \tau)  \geq 0 \Big]  - \mathbb{P}_{-w} \Big[\min_{1\leq i \leq m} \sum_{j=1}^{i} (\xi_j - \tau)  \geq 0 \Big]. \label{eq:density_int1} 
	\end{align} 
	where $\xi_j$ are iid Bernoulli random variables under both measures, with $\mathbb{P}_{z}[\xi_1=1] = F(z)$ and $\mathbb{P}_{-w}[\xi_1=1] = F(-w)$. 
	Next, define a stopping time $T = \inf \{ k \geq 1 : \sum_{j=1}^{k} (\xi_j - \tau) <0\}$. In turn, this implies, 
	\begin{align}
	&\mathbb{P}_{z} \Big[\min_{1\leq i \leq m} \sum_{j=1}^{i} (\xi_j - \tau)  \geq 0 \Big] = \mathbb{P}_z [ T > m ] , \nonumber \\
	&\mathbb{P}_{-w} \Big[\min_{1\leq i \leq m} \sum_{j=1}^{i} (\xi_j - \tau)  \geq 0 \Big] = \mathbb{P}_{-w} [ T > m ]. \nonumber 
	\end{align} 
	Plugging this back into \eqref{eq:density_int1}, we have, 
	\begin{align}
	\mathbb{P}\Big[0 < \max_{1\leq i \leq m} \tau(X_{1:i}) \leq z \Big] = \mathbb{P}_z [ T > m ] - \mathbb{P}_{-w} [ T > m ] = \mathbb{P}_{-w} [ T \leq m ] - \mathbb{P}_{z} [ T \leq m ]. \label{eq:density_int2}
	\end{align} 
	Continuing, we have, 
	\begin{align}
	\mathbb{P}_{-w} [ T \leq m ] - \mathbb{P}_{z} [ T \leq m ] = \mathbb{E}_{-w}\Big[ \mathbf{1}(T \leq m) \Big(1 - \frac{\mathrm{d}\mathbb{P}_z}{\mathrm{d}\mathbb{P}_{-w}} \Big) \Big]. 
	\end{align} 
	Further, 
	\begin{align}
	\frac{\mathrm{d}\mathbb{P}_z}{\mathrm{d}\mathbb{P}_{-w}} &= \frac{(F(z))^{\sum_{j=1}^{m} \xi_j } (1- F(z))^{m- \sum_{j=1}^{m} \xi_j}}{(F(-w))^{\sum_{j=1}^{m} \xi_j } (1- F(-w))^{m- \sum_{j=1}^{m} \xi_j}} \nonumber \\
	&=  \Big( \frac{F(z)/(1-F(z))}{F(-w)/(1-F(-w))} \Big)^{\sum_{j=1}^{m} (\xi_j - \tau)} \cdot \exp\Big[ m \Big( \tau \log \frac{F(z)}{F(-w)} + (1- \tau) \log \frac{1- F(z)}{1- F(-w)} \Big) \Big]. \nonumber 
	\end{align} 
	At this point, we choose $w$ such that 
	\begin{align}
	\tau \log \frac{F(z)}{F(-w)} + (1- \tau) \log \frac{1- F(z)}{1- F(-w)} =0 . \nonumber 
	\end{align} 
	To see that such a choice is indeed possible for $z$ sufficiently small, define a function $h(a) = \tau \log a + (1-\tau) \log (1-a)$. By direct computation, we note that $h$ is strictly concave, and attains it's maximum at $a = \tau$. Thus the function $h(F(t))$ is strictly increasing on $[-t_0,0]$ for some $t_0>0$. By the intermediate value theorem, there exists a unique $w>0$ such that $h(F(-w)) = h(F(z))$ for $z>0$ sufficiently small. Further, using Taylor series expansion near $a=\tau$, we have, 
	\begin{align}
	h(F(z)) = h(\tau) + \frac{h''(\kappa(z))}{2} (F(z) - \tau)^2, \,\,\,\, h(F(-w)) = h(\tau) + \frac{h''(\kappa(-w))}{2} (F(-w) - \tau)^2. \nonumber 
	\end{align} 
	Equating $h(F(-w)) = h(F(z))$, we have, 
	\begin{align}
	\Big( \frac{F(-w)- \tau}{F(z) - \tau} \Big)^2 = \frac{h''(\kappa(z))}{h''(\kappa(-w))}. \nonumber 
	\end{align} 
	
	By continuity, $\kappa(z), \kappa(-w) \to \tau$, and thus there exists a universal constant $C_1>0$ such that 
	\begin{align}
	| F(-w) - \tau | \leq C_1 |F(z) - \tau|. \nonumber 
	\end{align}
	
	Armed with this choice of $w$, we note that 
	\begin{align}
	\frac{\mathrm{d}\mathbb{P}_z}{\mathrm{d}\mathbb{P}_{-w}} &=  \Big( \frac{F(z)/(1-F(z))}{F(-w)/(1-F(-w))} \Big)^{\sum_{j=1}^{m} (\xi_j - \tau)} \nonumber 
	\end{align}
	is a martingale with respect to the canonical filtration under $\mathbb{P}_{-w}$, and therefore, 
	\begin{align}
	\mathbb{P}_{-w} [ T \leq m ] - \mathbb{P}_{z} [ T \leq m ] &= \mathbb{E}_{-w}\Big[ \mathbf{1}(T \leq m) \Big(1 - \frac{\mathrm{d}\mathbb{P}_z}{\mathrm{d}\mathbb{P}_{-w}} \Big) \Big]  \nonumber \\
	&=  \mathbb{E}_{-w}\Big[ \mathbf{1}(T \leq m)  \Big(1 - \Big( \frac{F(z)/(1-F(z))}{F(-w)/(1-F(-w))} \Big)^{\sum_{j=1}^{T} (\xi_j - \tau)} \Big) \Big]. \nonumber 
	\end{align} 
	We note that $\sum_{j=1}^{T} (\xi_j - \tau) <0$, and thus using the inequality $1-e^{-x} \leq x$ for $x \geq 0$, we have, 
	\begin{align}
	\mathbb{P}_{-w} [ T \leq m ] - \mathbb{P}_{z} [ T \leq m ] &\leq \mathbb{E}_{-w} \Big[ \mathbf{1}(T \leq m) \sum_{j=1}^{T} (\tau - \xi_j) \Big]  \log \Big( \frac{F(z)/(1-F(z))}{F(-w)/(1-F(-w))}\Big). \nonumber \\
	&\leq \log \Big( \frac{F(z)/(1-F(z))}{F(-w)/(1-F(-w))}\Big), \nonumber 
	\end{align} 
	where the last inequality follows upon observing that $\sum_{j=1}^{T} (\tau - \xi_j) \leq 1$. Finally, 
	\begin{align}
	\log \Big( \frac{F(z)/(1-F(z))}{F(-w)/(1-F(-w))}\Big) &= \log \Big( 1+ \frac{F(z) - F(-w)}{F(-w)} \Big)  - \log \Big(1- \frac{F(z) - F(-w)}{1- F(-w)} \Big) \nonumber \\
	&\lesssim (F(z) - F(-w)) \lesssim F(z) - \tau \leq C_2 z \nonumber 
	\end{align} 
	for some constant $C_2>0$. 
	This completes the proof. 
\end{proof}

Armed with Lemmas~\ref{lem:RW1},~\ref{lem:RW2}, we are now ready to prove Lemma~\ref{lem:woodroofe_meyer_argument}.

\begin{proof}[Proof of Lemma \ref{lem:woodroofe_meyer_argument}]
	Note that it suffices to establish this bound for $0<z\leq 1$---the bound for larger $z$ follows trivially by increasing $C$ if necessary. For ease of notation, denote $\max_{1\leq i \leq m} \tau(X_{1:i}) = Z_1$ and $\max_{1\leq j \leq n} (1- \tau)(-Y_{1:j}) = Z_2$, and let $G_1$ and $G_2$ denote the corresponding cdfs. We have, using $\{Z_1 + Z_2 \leq z\} \subset \{ \min\{Z_1, Z_2\} \leq z\}$, we have, 
	\begin{align}
	\mathbb{P}[Z_1 + Z_2 \leq z] \leq \mathbb{P}[Z_1 \leq z, Z_2 \leq z - Z_1] + \mathbb{P}[Z_2 \leq z, Z_1 \leq z - Z_2]. \nonumber 
	\end{align}
	We derive an upper bound on the first term. The second term is similar, and is thus omitted. We have, 
	\begin{align}
	\mathbb{P}[Z_1 \leq z, Z_2 \leq z - Z_1] &\leq \mathbb{P}[Z_1 \leq -1] + \mathbb{P}[-1 < Z_1 \leq z, Z_2 \leq z - Z_1]\nonumber \\
	&\leq \exp(-Cm) + \int_{-1}^{z} G_2(z-x) \mathrm{d}G_1(x) \nonumber \\
	&\leq \exp(-Cm) + \int_{-1}^{z} \Big( \frac{1}{\sqrt{n}} + C_1 (z-x) \Big) \mathrm{d}G_1(x) \nonumber \\
	&\leq \exp(-Cm) + \frac{G_1(z)}{\sqrt{n}} + C_1 \int_{-1}^{z} G_1(y) \mathrm{d}y \nonumber\\
	&\leq \exp(-Cm) + \frac{1}{\sqrt{n}} \Big( \frac{1}{\sqrt{m}} + C_2 z \Big) + \frac{C_1}{\sqrt{m}} \int_{-1}^{0} \exp(-C_2 m y^2) \mathrm{d}y +  C_1 \int_0^{z} \Big( \frac{1}{\sqrt{m}} + y \Big) \mathrm{d}y. \nonumber  \\
	&= \Theta\Big( \frac{1}{m} + \frac{1}{\sqrt{nm}} + \frac{z}{\sqrt{n}} + \frac{z}{\sqrt{m}} + z^2  \Big). \nonumber  
	\end{align}
	In the display above, the second inequality follows by using Lemma~\ref{lem:RW1} on the first term, the third inequality follows upon using Lemmas~\ref{lem:RW1},\ref{lem:RW2} on the second term, the fourth inequality follows using integration by parts. The fifth inequality follows by again using Lemmas~\ref{lem:RW1},\ref{lem:RW2} on the last two terms. The final conclusion follows from the inequality $2ab \leq a^2 + b^2$. 
\end{proof}

\section{Some auxilliary results} 

\begin{lemma}\label{lem:orderstat}
\label{lem:gaps} 
Let $U_1, \cdots, U_n \sim \Unif(0,1)$ be iid random variables, and let $0= U_{(0)} < U_{(1)} < \cdots < U_{(n)} <U_{(n+1)}=1$ denote the corresponding order statistics. Then there exists $C>0$ sufficiently large such that 
\begin{align}
\mathbb{P}\Big[ U_{(k+1)} - U_{(k)} < C \frac{\log n}{n} \,\, \mathrm{for\, all} \,\, 0\leq k \leq n \Big] \geq 1-n^{-2}.  \nonumber 
\end{align} 
\end{lemma}

\begin{proof}
Direct computation yields that for all $0 \leq k \leq n$, 
\begin{align}
\mathbb{P}[X_{(k+1)} - X_{(k)} > r ] = (1-r)^n \nonumber 
\end{align}
for all $r \in [0,1]$. By union bound, we have, 
\begin{align}
\mathbb{P}[U_{(k+1)} - U_{(k)} > C \frac{\log n}{n} \,\, \mathrm{for \, some} \,\, 0\leq k \leq n] \leq (n+1) \Big(1- C \frac{\log n}{n}\Big)^n \leq \frac{1}{n^2}  \nonumber 
\end{align}
for $C>0$ is sufficiently large. This completes the proof. 
\end{proof}

\begin{remark}
\label{rem:gaps_union}
We will apply Lemma \ref{lem:gaps} to the setting where $U_1, \cdots, U_n \sim \Unif(A)$ i.i.d., where $A \subseteq \mathbb{R}$ is a disjoint union of intervals. We let $ \min\{x \in A\} = U_{(0)} < U_{(1)}< \cdots < U_{(n)} < U_{(n+1)} =\max\{x \in A\}$ denote the corresponding order statistics. Finally, one can construct an interval of length $\mathrm{Leb}(A)$ by ``arranging" its constituent intervals in a contiguous fashion; for $x,y \in A$, let $d_A(x,y)$ denote the distance between $x,y$ on this re-arranged interval. With this convention, one can derive the following immediate corollary of Lemma \ref{lem:gaps}. 
\begin{align}
    \mathbb{P}\Big[d_A(U_{(k+1)}, U_{(k)}) < C \frac{ \log n}{n} \mathrm{for\, all}\, 0 \leq k \leq n \Big] \geq 1 - \frac{1}{n^2}. \nonumber 
\end{align}
We will use Lemma \ref{lem:gaps} in this specific form. 
\end{remark}

\begin{lemma}
\label{lem:bin_dev} 
Let $X \sim \mathrm{Bin}(n,p)$. For any $\delta>0$, there exists $C:=C(\delta)$ such that $\mathbb{P}(|X- np| > \delta np ) \leq 2 \exp(-C np)$. 
\end{lemma}

\begin{proof}
This will follow by a direct application of Bernstein's inequality. We have, 
\begin{align*}
\mathbb{P}(|X- np| > \delta np ) \leq 2 \exp\Big( - \frac{\delta^2 n^2 p^2}{np(1-p) + \frac{1}{3} \delta np} \Big)\leq 2 \exp(- C np). 
\end{align*}
\end{proof}

\bibliographystyle{plain}
\bibliography{references}

\end{document}